\documentclass[11pt]{amsart}
\usepackage{geometry,ulem}
\usepackage{amscd,amssymb,verbatim,xcolor,mathrsfs}
\usepackage{longtable, multirow}
\usepackage{soul,cancel}
\usepackage{xcolor}
\usepackage{tikz} 
\usetikzlibrary {arrows.meta,bending}
\date{\today}

\newcommand\pmat{\begin{pmatrix}}
\newcommand\epmat{\end{pmatrix}}

\newcommand\frg{{\mathfrak g}}
\newcommand\frh{{\mathfrak h}}
\newcommand\frk{{\mathfrak k}}

\newcommand\frp{{\mathfrak p}}
\newcommand\frq{{\mathfrak q}}

\newcommand\frt{{\mathfrak t}}

\newcommand\Ad{{\operatorname{Ad}}}
\newcommand\ad{{\operatorname{ad}}}

\newcommand\Ker{{\operatorname{Ker}}}

\usepackage{amsfonts}
\usepackage{amsmath}
\usepackage{indentfirst}
\usepackage{amsthm}

\usepackage{longtable}
\usepackage{booktabs}

\usepackage{verbatim}

\DeclareMathOperator{\Ima}{Im}

\DeclareMathOperator{\DIM}{dim}
\newcommand{\mf}[1]{\mathfrak{#1}}
\newcommand{\mb}[1]{\mathbb{#1}}

\newcommand{\bp}{\begin {proof}}

\newcommand{\ba}{\begin {align*}}
\newcommand{\ea}{\end {align*}}

\theoremstyle{plain}
\newtheorem{thm}{Theorem}[section]

\newtheorem{cor}[thm]{Corollary}

\newtheorem{prop}[thm]{Proposition}

\newtheorem{lemma}[thm]{Lemma}
\theoremstyle{definition}
\newtheorem{defi}[thm]{Definition}

\newtheorem{rmk}[thm]{Remark}

\def\dim{\mathop{\hbox {dim}}\nolimits}
\def\Ad{\mathop{\hbox {Ad}}\nolimits}
\def\ad{\mathop{\hbox {ad}}\nolimits}

\newcommand{\Kt}{\widetilde K}
\newcommand{\wt}{\widetilde}

\newcommand{\DI}{{\mathrm{DI}}}

\newcommand{\ms}{\medskip}

\newcommand{\norm}[1]{\left\lVert#1\right\rVert}

\newcommand{\hf}{\frac{1}{2}}

\usepackage[all]{xy}

\begin{document}
\title[Dirac cohomology of minimal representations]{Dirac cohomology of minimal representations}
\author{Xuanchen Zhao}

\address[Zhao]{School of Science and Engineering, The Chinese University of Hong Kong, Shenzhen,
Guangdong 518172, P. R. China}
\email{116010313@link.cuhk.edu.cn}

\begin{abstract}
In this paper, we study the Dirac cohomology of minimal representations for all real reductive groups $G$. The Dirac indices of these representations are also studied when $G$ is of equal rank, giving some counterexamples of a conjecture of Huang proposed in 2015.
\end{abstract}

\maketitle
\setcounter{tocdepth}{1}

\section{Introduction}\label{sec:intro}

The notion of Dirac operator plays an important role in representation theory
of real reductive groups. In the 1970s, Parthasarathy \cite{P1, P2} and Schmid used Dirac operators to give a geometric realization of the discrete series. Later, Vogan in \cite{V2} introduced the notion of {\bf Dirac cohomology} for irreducible representations in the early 1990s, and formulated a conjecture on its relationship with the infinitesimal character of the representation. It was subsequently proven by Huang and Pand\v zi\'c in \cite{HP1}. 

\medskip
One application of Dirac cohomology is to have a better understanding of the unitary dual $\widehat{G}$.
Indeed, the set of unitary $(\mathfrak{g},K)-$modules with non-zero Dirac cohomology
(the {\bf Dirac series}) $\widehat{G}^d$ contains a large example of interesting unitary 
representations.
For instance, \cite{HP2} showed that
$\widehat{G}^d$ contains all unitary modules with non-zero $(\mathfrak{g},K)-$cohomology. In particular, this implies $\widehat{G}^d$ contains all $A_\frq(\lambda)$ modules in good range by \cite{VZ}. On the other hand, \cite{BP1} and \cite{BP2} classified all unipotent representations appearing in $\widehat{G}^d$ for all complex groups and some real reductive groups.

\medskip
In this paper, we will concentrate on minimal representations. They are firstly studied by Vogan \cite{V3}. For $G$ not of type $A$, the minimal representations are irreducible admissible representations whose Harish-Chandra modules are annihilated exactly by the Joseph ideal \cite{J1}. They are also seen as unipotent representations attached to the minimal nilpotent orbit. The Harish-Chandra module of a minimal representation is called a minimal $(\frg,K)$-module. Indeed, all minimal $(\mf{g},K)$-modules are unitarizable, and have associated variety equal to the closure of the minimal nilpotent orbit. Recently, a full classification of minimal $(\frg,K)$-modules of real simple Lie groups is given by Tamori \cite{T1}. For any simply connected simple real Lie group $G$ (which will be represented by its Lie algebra $\mf{g}_0:=Lie(G)$ in this paper), the number of minimal $(\frg,K)$-modules is given in Table \ref{tab:nr-minimal-repns}. 

\renewcommand\arraystretch{1.35}
\begin{table} 
    \centering
    \begin{tabular}{|l|l|} \hline
    $\mf{g}_0$ & Number of minimal $(\mf{g},K)$-modules \\ \hline
    $\mf{sp}(2n,\mb{R})$ $(n\geq2)$ & 4 \\ \hline
    $\mf{so}(p,2)$ $(p\geq5)$, $\mf{so}^*(2n)$ $(n\geq4)$, & \multirow{2}{*}{2} \\
    $\mf{e}_{6(-14)}$, $\mf{e}_{7(-25)}$, $\mf{sp}(n,\mb{C})$ $(n\geq2)$ & \\ \hline
    $\mf{so}(p,q)$ $(p,q\geq3,~p+q\geq8~even)$, & \multirow{4}{*}{1} \\
    $\mf{so}(p,3)$ $(p\geq4~even)$, $\mf{e}_{6(6)}$, $\mf{e}_{6(2)}$, & \\ 
    $\mf{e}_{7(7)}$, $\mf{e}_{7(-5)}$, $\mf{e}_{8(8)}$, $\mf{e}_{8(-24)}$, $\mf{f}_{4(4)}$, $\mf{g}_{2(2)}$  & \\ 
    $\mf{so}(n,\mb{C})$ $(n\geq 7)$, $\mf{e}_6(\mb{C})$, $\mf{e}_7(\mb{C})$, $\mf{e}_8(\mb{C})$, $\mf{f}_4(\mb{C})$, $\mf{g}_2(\mb{C})$ & \\ \hline
    \end{tabular}
    \caption{Number of minimal $(\mf{g},K)$-modules (where $K$ is simply connected) for each real simple Lie algebra $\mf{g}_0$ such that $\mf{g}$ is not of the type $A$. Cited from \cite{T1}.}  \label{tab:nr-minimal-repns}
\end{table}
\renewcommand\arraystretch{1}

\medskip
Here is a summary of some known results for the Dirac cohomology of minimal $(\mf{g},K)$-modules. For complex groups, the Dirac cohomology of minimal $(\frg,K)$-modules are nonzero only when $\frg_0$ is isomorphic to some $\mf{sl}(2n,\mb{C})$, $\mf{so}(2n+1,\mb{C})$ or $\mf{sp}(2n,\mb{C})$ (\cite{BP1}, \cite{DW1}). 
For real classic groups, when $\mf{g}_0$ is $\mf{sp}\left(2n,\mb{R}\right)$, $\mf{so}^*(2n)$ or $\mf{u}(p,q)$, the Dirac cohomology of Wallach representations (which include the minimal representations) has been given in \cite{HPP}. 
For real exceptional groups, the Dirac cohomology of minimal $(\mf{g},K)$-modules is studied in \cite{D1} for $\mf{g}_0=\mf{e}_{6(2)}$, \cite{DDL} for $\mf{g}_0=\mf{e}_{7(-5)}$ and \cite{DDW} for $\mf{g}_0=\mf{e}_{7(7)}$. It can be found in \cite{DDY} that the minimal representation of $\mf{g}_0=\mf{e}_{6\left(6\right)}$ is in the Dirac series while this is not true when $\mf{g}_0=\mf{g}_{2\left(2\right)}$.

\medskip
In this paper, we complete the study of the Dirac cohomology of minimal $(\mf{g},K)$-modules for all connected real reductive Lie groups. By the discussions above, we need only to focus on the following choices of $\mf{g}_0$:
\begin{align}\label{list-of-g0}
\mf{so}(p,q)~(p\geq q\geq3,\ p+q \geq 8\ \text{even}),~\mf{so}(p,3)~(p\geq4 \ \text{even}),~\mf{f}_{4(4)},~\mf{e}_{8(8)},~\mf{e}_{8(-24)}.
\end{align}
Apart from computing the Dirac cohomology, we will also study the Dirac index (see Section \ref{subsec:di}) when $\mf{g}_0$ is of equal rank.
In particular, if $\frg_0$ is one that $G/K$ not Hermitian symmetric, then from Table \ref{tab:nr-minimal-repns} we may notice that there is at most one minimal $\left(\frg,K\right)$-module which will be denoted by $V_{min}(\frg_0)$ if exists. For these choices of $\frg_0$, we will show that the Dirac index (if exists) is zero if and only if its infinitesimal character is singular (see Table \ref{tab:summary} for a summary of the results). This will give a few more counter-examples on Conjecture 10.3 of \cite{H1}.

\medskip
The paper is organized as follows: Section 2 recalls some preliminary results on Dirac cohomology, Dirac index and minimal $(\frg,K)$-modules. For an irreducible unitarizable $(\frg,K)$-module $V$ (such as a minimal $(\frg,K)$-module), we introduce the notion of \textbf{Dirac triples}, so that the computation of the Dirac cohomology $H_D(V)$ is reduced to finding all Dirac triples of $V$.
In Sections 3--8, we compute and list explicitly the Dirac triples of $V_{min}(\frg_0)$ for $\mf{g}_0 = \mf{so}(2m,2n)$, $\mf{so}(2m+1,2n+1)$, $\mf{so}(2n,3)$, $\mf{f}_{4(4)}$, $\mf{e}_{8(8)}$ and $\mf{e}_{8(-24)}$. Their Dirac cohomologies and Dirac indices are also given, verifying the results in Table \ref{tab:summary}. We state in Section 9 a corollary on Dirac triples. Finally, in Appendix \ref{sec-python}, we give the \texttt{python} codes used to compute all Dirac triples when $\mf{g}_0$ is of Type $E_8$.

\renewcommand\arraystretch{1.25}
\begin{table}
    \centering
    \begin{tabular}{|c|c|c|c|} \hline
    $\mf{g}_0$ & Dirac cohomology & Dirac index & Infinitesimal character \\ \hline
    $\mf{so}\left(2m,2n\right)$ $(m\geq n\geq2)$ & Nonzero & Zero & Singular \\ \hline
    $\mf{so}\left(2m+1,2n+1\right)$ & \multirow{2}{*}{Nonzero} & \multirow{2}{*}{N/A} & \multirow{2}{*}{Singular} \\ 
    $(m\geq2,~n\geq1,~m\geq n)$ & & & \\ \hline
    $\mf{so}\left(2n,3\right)$ $(n\geq2)$ & Nonzero & Nonzero & Regular \\ \hline
    $\mf{f}_{4\left(4\right)}$ & Nonzero & Nonzero & Regular \\ \hline
    $\mf{e}_{8\left(8\right)}$ & Nonzero & Zero & Singular \\ \hline
    $\mf{e}_{8\left(-24\right)}$ & Nonzero & Zero & Singular \\ \hline
    $\mf{e}_{6\left(6\right)}$ & Nonzero & N/A & Singular \\ \hline
    $\mf{e}_{6\left(2\right)}$ & Nonzero & Zero & Singular \\ \hline
    $\mf{e}_{7\left(7\right)}$ & Nonzero & Zero & Singular \\ \hline
    $\mf{e}_{7\left(-5\right)}$ & Nonzero & Zero & Singular \\ \hline
    $\mf{g}_{2\left(2\right)}$ & Zero & Zero & Regular \\ \hline
    \end{tabular} 
    \caption{Information about the minimal $(\frg,K)$-modules when $G/K$ is not Hermitian symmetric} \label{tab:summary}
\end{table}
\renewcommand\arraystretch{1}

\section{Preliminaries}\label{sec:preliminary}
Let $G$ be a connected real reductive Lie group and $\mf{g}_0:=Lie(G)$, $\mf{g}:=\mf{g}_0\otimes_\mb{R}\mb{C}$. Fix a Cartan involution $\theta$ of $\mf{g}_0$ (we also use $\theta$ to refer to the complex linear extension of $\theta$ on $\mf{g}$), denote by $\mf{g}_0=\mf{k}_0\oplus\mf{p}_0$ the corresponding Cartan decomposition and by $\mf{g}=\mf{k}\oplus\mf{p}$ its complexification. The connected subgroup $K$ of $G$ such that $Lie(K)=\mf{k}_0$ is a maximal connected compact subgroup of $G$.
Notice that in this paper $G$ is always assumed to be connected and simply connected, so $G$ is uniquely determined by $\mf{g}_0$.
We fix a nondegenerate $\text{Ad}(G)$-invariant symmetric bilinear form on $\frg$ and denote it by $\langle\cdot,\cdot\rangle$. Let $\langle\cdot,\cdot\rangle^*$ be the coadjoint-invariant symmetric bilinear form on $\frg^*$ induced by $\langle\cdot,\cdot\rangle$. Set $\norm{\cdot}$ to be the norm induced by $\langle\cdot,\cdot\rangle^*$.

\medskip
Let $\mf{t}$ be a Cartan subalgebra of $\mf{k}$ and $\mf{a}$ the centralizer of $\mf{t}$ in $\mf{p}$. Then $\mf{h}:=\mf{t}\oplus\mf{a}$ is a maximally compact $\theta$-stable Cartan subalgebra of $\mf{g}$. Denote by $\Phi(\mf{g},\mf{h})$ and $\Phi(\mf{k},\mf{t})$ the root systems. 
Fix a positive root system $\Phi^+(\mf{g},\mf{h})$ (resp.$\Phi^+(\mf{k},\mf{t})$), denote its half sum by $\rho_\frg$ (resp. $\rho_\frk$) and denote the corresponding set of simple roots by $\Pi(\frg,\frh)$ (resp. $\Pi(\frk,\frt)$) . 
When choosing the two positive root systems, we require that they are compatible, i.e. $\Phi^+(\mf{k},\mf{t})\subseteq\Phi^+(\mf{g},\mf{h})$.
Regard $\frt^*$ and $\frh^*$ as subspaces of $\frg^*$ in the natural way. A vector $v\in\mf{h}^*$ is said to be \textbf{K-dominant} if $\langle v,\alpha\rangle^*\geq0$ for every $\alpha\in\Phi^+(\mf{k},\mf{t})$. Let $W(\mf{g},\mf{h})$ and $W(\mf{k},\mf{t})$ be Weyl groups.
Denote by $R(\mf{g})$ the $W(\mf{g},\mf{h})$-orbit of $\rho_\frg$, and by $R_K(\mf{g})$ the K-dominant subset of $R(\mf{g})$.
If $V$ is a $(\frg,K)$-module, denote by $\Lambda(V)$ the set of representatives of the infinitesimal character of $V$, and by $\Lambda_K(V)$ the K-dominant subset of $\Lambda(V)$.

\medskip
Let $\wt{K}$ be the spin double cover of $K$. Namely,
$$
\widetilde{K}:=\{ (k,s)\in K\times {\rm Spin}(\frp_0)\ :\ \Ad (k)=p(s)\},
$$
where $p: \text{Spin}(\frp_0)\rightarrow \text{SO}(\frp_0)$ is the spin double covering map. Denote by $E_\mu$ (resp. $\wt{E}_\mu$) the $K$-type (resp. $\wt{K}$-type) with highest weight $\mu$. 
If $V$ is a $K$(or $\wt{K}$)-representation, denote by $V_\mu$ the isotypic component of $V$ with highest weight $\mu$.

\subsection{Dirac cohomology}
We recall the construction of Dirac operator and Dirac cohomology.
Notice that $\langle\cdot,\cdot\rangle$ is positive definite on $\frp_0$.
Fix an orthonormal basis $Z_1,\cdots,Z_n$ of $\mf{p}_0$ with respect to $\langle\cdot,\cdot\rangle$.
Let $U(\frg)$ be the universal enveloping algebra of $\frg$, and $C(\frp)$
be the Clifford algebra of $\frp$ with respect to $\langle\cdot,\cdot\rangle$.
The {\bf Dirac operator}
$D\in U(\frg)\otimes C(\frp)$ is defined as
$$D:=\sum_{i=1}^{n}\, Z_i \otimes Z_i.$$
$D$ does not depend on the choice of the
orthonormal basis $Z_i$ and  is $K-$invariant for the diagonal
action of $K$ induced by the adjoint actions on both factors.

\medskip
Define $\Delta: \frk \to U(\frg)\otimes C(\frp)$ by $\Delta(X)=X\otimes
1 + 1\otimes \alpha(X)$, where $\alpha:\frk\to C(\frp)$ is the
composition of $\ad:\frk\longrightarrow \mathfrak{so}(\frp)$ with the embedding  $\mathfrak{so}(\frp)\cong\wedge^2(\frp)\hookrightarrow C(\frp)$. Write $\frk_{\Delta}:=\Delta(\frk)$, and denote by $\Omega_{\frg}$ (resp. $\Omega_{\frk}$)
the Casimir operator of $\frg$ (resp. $\frk$). $\Delta$ can be uniquely extended as an algebra homomorphism $\Delta: U(\frk) \to U(\frg)\otimes C(\frp)$. Let $\Omega_{\frk_{\Delta}}$ be the image of $\Omega_{\frk}$ under $\Delta$. Then \cite{P1} gives:
\begin{equation}\label{eq-D^2}
D^2=-\Omega_{\frg}\otimes 1 + \Omega_{\frk_{\Delta}} + (\|\rho_\frk\|^2-\|\rho_\frg\|^2) 1\otimes 1.
\end{equation}

If $V$ is an irreducible admissible $(\frg,K)$-module, and if $S$ denotes a spin module for $C(\frp)$, then $V\otimes S$ is a $(U(\frg)\otimes C(\frp), \widetilde{K})$ module. The action of $U(\frg)\otimes C(\frp)$ is the obvious one, and $\widetilde{K}$ acts on both factors; on $V$ through $K$ and on $S$ through the spin group $\text{Spin}\,{\frp_0}$.

As discussed above, the Dirac operator $D$ acts on $V\otimes S$, and the Dirac cohomology of $V$ is defined as the $\widetilde{K}$-module
\begin{equation}\label{def-Dirac-cohomology}
H_D(V)=\text{Ker}\, D/ (\text{Im} \, D \cap \text{Ker} D).
\end{equation}
Both $\Ker(D)$ and $\Ima(D)\cap\Ker(D)$ are $\wt{K}$-submodules of $V\otimes S$, so $H_D(V)$ can also be regarded as a $\wt{K}$-submodule of $V\otimes S$.

The following foundational result on Dirac cohomology, conjectured
by Vogan,  was proven by Huang and Pand\v zi\'c in 2002:

\begin{thm}[\cite{HP1} Theorem 2.3]\label{thm-HP}
Let $V$ be an irreducible $(\frg, K)$-module.
Assume that the Dirac
cohomology of $V$ is nonzero, and that it contains the $\widetilde{K}$-type with highest weight $\gamma\in\frt^{*}\subset\frh^{*}$. Then the infinitesimal character of $V$ is conjugate to
$\gamma+\rho_{c}$ under $W(\frg,\frh)$.
\end{thm}

\subsection{Dirac cohomology for unitary $(\frg,K)$-modules}

A $(\mf{g},K)$-module $V$ is called \textbf{unitary} (or \textbf{unitarizable}) if there is a positive definite Hermitian form $\langle\cdot,\cdot\rangle_V$ on $V$ such that
\begin{align*}
\langle Xv,w\rangle_V&=-\langle v,Xw\rangle_V\\
\langle kv,kw\rangle_V&=\langle v,w\rangle_V
\end{align*}
for any $X\in\mf{g}$, $k\in K$ and $v,w\in V$.

\medskip
From now on, we focus on the case when $V$ is an irreducible unitarizable $(\frg,K)$-module. 
In such a case we have $H_D(V)=\text{Ker}(D)=\text{Ker}(D^2)$ (see Section 3.2 of \cite{HP2}), then \eqref{eq-D^2} implies that $D^2$ acts on $(V\otimes S)_\delta$ by scalar multiplication by $\norm{\delta+\rho_\frk}^2-\norm{\lambda}^2$. 
Actually, Theorem \ref{thm-HP} implies that $H_D(V)$ is exactly the direct sum of isotypic components $(V\otimes S)_\delta$ such that
\begin{equation}\label{eq-HP}
\delta+\rho_\frk=\lambda
\end{equation}
for some $\lambda\in\Lambda(V)$.

\medskip
To put \eqref{eq-HP} in practice, note that 
\begin{equation}
    S=\bigoplus_{\rho\in R_K(\frg)}2^{[\hf\DIM\mf{a}]}\widetilde{E}_{\rho-\rho_\frk}
\end{equation}
where $[\hf\DIM\mf{a}]$ means the smallest integer which is larger than or equal to $\hf\DIM\mf{a}$.
Let $V_\mu$ (resp. $S_{\rho-\rho_\frk}$) be the isotypic component of $V$ (resp. $S$) with highest weight $\mu$ (resp. $\rho-\rho_\frk$).
If $v\in\frt^*\subseteq\frh^*$, denote by $\{v\}$ the only K-dominant element in the $W(\mf{k},\mf{t})$-orbit of $v$.
Then the isotypic component $(V\otimes S)_{\{\mu-\rho+\rho_\frk\}}$ is called the \textbf{PRV-component} of $V_\mu\otimes S_{\rho-\rho_\frk}$. Among all isotypic components of $V_\mu\otimes S_{\rho-\rho_\frk}$, the PRV-component has the shortest infinitesimal character (see Section 9.1 of \cite{W1}). 

\medskip
If $V$ is an irreducible unitary representation, then \textbf{Parthasarathy's Dirac operator inequality} states that every $\wt{K}$-type $\wt{E}_\delta$ that appears in $V\otimes S$ must satisfy
\begin{equation}\label{eq-Par}
\norm{\delta+\rho_\frk}\geq\norm{\lambda}
\end{equation}
for any $\lambda\in\Lambda(V)$.
$V\otimes S$ can be decomposed as the direct sum of $V_\mu\otimes S_{\rho-\rho_\frk}$ over $\mu$ and $\rho$. Combine (\ref{eq-Par}) with (\ref{eq-HP}) and the property of the PRV-component, we know that $H_D(V)$ is the direct sum of some isotypic components of the form $(V\otimes S)_{\{\mu-\rho+\rho_\frk\}}$.

\medskip
The multiplicity $m(\wt{E}_{\lambda-\rho_\frk}:H_D(V))$ of $\wt{E}_{\lambda-\rho_\frk}$ in $H_D(V)$ is 
\begin{equation}\label{eq-multiplicity}
    m(\wt{E}_{\lambda-\rho_\frk}:H_D(V))=\sum_{\{\mu-\rho+\rho_\frk\}=\lambda-\rho_\frk}m(E_\mu:V)m(\wt{E}_{\rho-\rho_\frk}:S) = 2^{[\hf\DIM\mf{a}]}\sum_{\{\mu-\rho+\rho_\frk\}=\lambda-\rho_\frk}m(E_\mu:V).
\end{equation}
where the sum is taken over $\rho\in R_K(\frg)$, $\lambda\in\Lambda_K(V)$ and all highest weights $\mu$ of $K$-types of $V$. 

\medskip
We can reformulate the question of computing $H_D(V)$ explicitly for a unitary $V$ as the question of looking for all triples $(\mu,\rho,\lambda)$ where $\rho\in R_K(\frg)$, $\lambda\in\Lambda_K(V)$ and $E_\mu$ is a $K$-type of $V$, such that
\begin{equation}\label{eq-DT}
    \{\mu-\rho+\rho_\frk\}=\lambda-\rho_\frk.
\end{equation}
We call such a $(\mu,\rho,\lambda)$ a \textbf{Dirac triple} (or say that such a $(\mu,\rho,\lambda)$ is \textbf{Dirac}).

\subsection{Dirac index}\label{subsec:di}
When $rank(G)=rank(K)$, we can define the \textbf{Dirac index} $\DI(V)$ of an irreducible $(\mf{g},K)$-module $V$.

\medskip
Fix $\rho_\frg\in R(\mf{g})$, then for every $\rho\in R(\mf{g})$ there is a unique $w\in W(\mf{g},\mf{h})$ such that $\rho=w\rho_\frg$. We can divide $R_K(\mf{g})$ into the union of $R_K^+(\mf{g})$ and $R_K^-(\mf{g})$ by the parity of the length of $w$. Namely, for $w\rho_\frg \in R_K^+(\mf{g})$, the length of $w$ is even, and for $w\rho_\frg \in R_K^-(\mf{g})$, the length of $w$ is odd. 

\medskip
As discussed above, $S=\bigoplus_{\rho\in R_K(\mf{g})}2^{[\hf\dim\mf{a}]}\wt{E}_{\rho-\rho_\frk}$ as a $\wt{K}$-module. Let $$S^\pm:=\bigoplus_{\rho\in R_K^\pm(\mf{g})}2^{[\hf\dim\mf{a}]}\wt{E}_{\rho-\rho_\frk},$$ then both $S^+$ and $S^-$ are $\wt{K}$-invariant. 
The action of $D$ interchanges $V\otimes S^+$ and $V\otimes S^-$:
\begin{displaymath}
    \xymatrix{
        V\otimes S^+ \ar@/^/[r]^{D}  & V\otimes S^- \ar@/^/[l]^{D} 
    }
\end{displaymath}
Let $D^\pm:=D|_{V\otimes S^\pm}$ and let
\begin{equation}
H_D(V)^\pm:=\frac{\Ker(D^\pm)}{\Ker(D^\mp)\cap\Ima(D^\mp)},
\end{equation}
then $H_D(V)=H_D(V)^+\oplus H_D(V)^-$.
The \textbf{Dirac index} of $V$, denoted by $\DI(V)$, is defined as the virtual $\wt{K}$-module $H_D^+(V)-H_D^-(V)$.

\medskip
As we have mentioned, $H_D(V)$ is the direct sum of some $\wt{K}$-isotypic components of $V\otimes S$. More explicitly, for each Dirac triple $(\mu,\rho,\lambda)$, the the PRV-component of $V_\mu\otimes S_{\rho-\rho_\frk}$ contributes to $H_D(V)$. Therefore, it is easy to see that 
for any Dirac triple $(\mu,\rho,\lambda)$, 
$\wt{E}_{\lambda-\rho_\frk}$ belongs to $H_D^\pm(V)$ if and only if $\rho\in R_K^\pm(\mf{g})$.

\medskip
To determine whether $\rho$ is in $R_K^+(\mf{g})$ or it is in $R_K^-(\mf{g})$, we need the following proposition:
\begin{prop}[\cite{Kn} Proposition 4.11]\label{prop-length-of-w}
Suppose $\rho_\frg=\hf\sum_{\alpha\in\Phi^+(\mf{g},\mf{h})}\alpha$ and $\rho=w\rho_\frg$, then the length of $w$ is equal to the number of $\alpha\in\Phi^+(\mf{g},\mf{h})$ such that $\left\langle\alpha,\rho\right\rangle^*<0$. In particular, if $\rho$ is K-dominant, the length of $w$ is equal to the number of $\alpha\in\Phi^+(\mf{g},\mf{h})-\Phi^+(\mf{k},\mf{t})$ (recall that $\mf{h}=\mf{t}$ when $rank(K)=rank(G)$) such that $\left\langle\alpha,\rho\right\rangle^*<0$. 
\end{prop} 

\medskip
To sum up, we have the following two equations by which we can compute Dirac cohomology and Dirac index explicitly using the data of Dirac triples:
\begin{equation}\label{eq-dc}
H_D(V)=\bigoplus_{(\mu,\rho,\lambda)\text{ is Dirac}}2^{[\hf\dim\mf{a}]}m(E_\mu;V)\wt{E}_{\lambda-\rho_\frk};
\end{equation}
\begin{equation}\label{eq-di}
\DI(V)=\bigoplus_{(\mu,\rho,\lambda)\text{ is Dirac}\atop \rho\in R_K^+(\mf{g})}2^{[\hf\dim\mf{a}]}m(E_\mu;V)\wt{E}_{\lambda-\rho_\frk}-\bigoplus_{(\mu,\rho,\lambda)\text{ is Dirac}\atop \rho\in R_K^-(\mf{g})}2^{[\hf\dim\mf{a}]}m(E_\mu;V)\wt{E}_{\lambda-\rho_\frk}.
\end{equation}

\subsection{Minimal representations} \label{subsec:min-repn}

Let $G$ be a simply connected real reductive Lie group {\bf not} of type $A$, then an irreducible $(\frg,K)$-module is said to be \textbf{minimal} if its its annihilator ideal is the Joseph ideal \cite{J1}. In particular, its infinitesimal character is listed in Table \ref{tab:infi-char} below:
\renewcommand\arraystretch{1.35}
\begin{table}[!h]
    \centering
    \begin{tabular}{|c|c|} \hline
    {$\frg$} &  Infinitesimal character (in terms of fundamental weights)  \\ \hline
    $\mf{so}(2n+1,\mb{C})$ $(n\geq3)$ & $\sum_{i=1}^{n-3}\omega_i+\hf\omega_{n-2}+\hf\omega_{n-1}+\omega_{n}$ \\
    $\mf{sp}(n,\mb{C})$ $(n\geq2)$ & $\sum_{i=1}^{n-1}\omega_i+\hf\omega_{n}$ \\
    $\mf{so}(2n,\mb{C})$ $(n\geq4)$ & $\sum_{i=1}^{n-3}\omega_i+\omega_{n-1}+\omega_{n}$ \\
    $\mf{e}_n(\mb{C})$ $(n=6,7,8)$ & $\sum_{i=1}^{n}\omega_i-\omega_{4}$ \\
    $\mf{f}_4(\mb{C})$ & $\hf\omega_1+\hf\omega_2+\omega_3+\omega_{4}$ \\
    $\mf{g}_2(\mb{C})$ & $\omega_1+\frac{1}{3}\omega_2$ \\ \hline
    \end{tabular}
    \caption{Infinitesimal characters of minimal $(\frg,K)$-modules when $\frg_0$ is not of type $A$. Cited from \cite{T1}.}
    \label{tab:infi-char}
\end{table}
\renewcommand\arraystretch{1}

It is well known that minimal representations have the smallest positive Gelfand-Kirillov dimension. More explicitly, in \S5 of \cite{V3}, Vogan proved the following statement:
\begin{thm}\label{thm-min-repn}
If a $(\frg,K)$-module $V$ is minimal, then:
\begin{enumerate}
    \item there is a $K$-type $E_{\mu_0}$ and \textbf{a} highest weight $\beta$ of the $\frk$-module $\frp$ such that the $K$-types of $V$ are $\{E_{\mu_0+s\beta}|s=0,1,\cdots\}$;
    \item each $K$-type in $(1)$ occurs in $V$ with multiplicity one. 
\end{enumerate}
\end{thm}

\medskip
By (2) of Theorem \ref{thm-min-repn}, when $V$ is a minimal $(\frg,K)$-module, (\ref{eq-dc}) and (\ref{eq-di}) are turned into
\begin{equation}\label{eq-dc-min-repn}
H_D(V)=\bigoplus_{(\mu,\rho,\lambda)\text{ is Dirac}}2^{[\hf\dim\mf{a}]}\wt{E}_{\lambda-\rho_\frk}.
\end{equation}
and
\begin{equation}\label{eq-di-min-repn}
\DI(V)=\bigoplus_{(\mu,\rho,\lambda)\text{ is Dirac}\atop \rho\in R_K^+(\mf{g})}2^{[\hf\dim\mf{a}]}\wt{E}_{\lambda-\rho_\frk}-\bigoplus_{(\mu,\rho,\lambda)\text{ is Dirac}\atop \rho\in R_K^-(\mf{g})}2^{[\hf\dim\mf{a}]}\wt{E}_{\lambda-\rho_\frk}.
\end{equation}

To apply (\ref{eq-dc-min-repn}) and (\ref{eq-di-min-repn}) we need the data of all Dirac triples of $V$. When $G$ is the connected and simply connected reductive Lie group whose Lie algebra is in \eqref{list-of-g0}, we compute and list in this paper all the Dirac triples for every minimal $(\frg,K)$-module of $G$. Actually, according to Table \ref{tab:nr-minimal-repns}, for each $G$ we considered in this paper, there is only one minimal $(\frg,K)$-module, up to isomorphism.

\section{$\mf{g}_0=\mf{so}(2m,2n)$ $(m\geq n\geq2)$}\label{sec:sopqeven}
Let $\mf{g}_0=\mf{so}(2m,2n)$, where $m\geq n\geq 2$. We make a choice of compatible positive root systems $\Phi^+(\frg,\frh)$ and $\Phi^+(\frk,\frt)$ as following.

\renewcommand\arraystretch{1.2}
\begin{center}
\begin{tabular}{c|c}\hline
$\mf{g}$ & $\mf{so}(2m+2n,\mb{C})$ \\ 
$\mf{k}$ & $\mf{so}(2m,\mb{C})\oplus\mf{so}(2n,\mb{C})$ \\ 
$\Pi(\mf{g},\mf{h})$ &  $\{e_i-e_{i+1}|i=1,\cdots,m+n-1\}\cup\{e_{m+n-1}+e_{m+n}\}$ \\
$\Pi(\mf{k},\mf{t})$ &  $\{e_i-e_{i+1}|i=1,\cdots,m+n-1,~i\neq m\}\cup\{e_{m-1}+e_m,~e_{m+n-1}+e_{m+n}\}$ \\
$\rho_\frg\in R_K(\mf{g})$ & $(m+n-1,m+n-2,\cdots,1,0)$ \\ 
$\rho_\frk$ & $(m-1,m-2,\cdots,1,0;n-1,n-2,\cdots,1,0)$  \\ \hline
\end{tabular}
\end{center}
\renewcommand\arraystretch{1}

By Table \ref{tab:nr-minimal-repns}, there is only one minimal $(\mf{g},K)$-module $V:=V_{min}(\mf{so}(2m,2n))$ with infinitesimal character $$\lambda_0 = (m+n-2,m+n-3,\cdots,2,1,1,0) \in \Lambda_K(V).$$ 
Also, the $K$-types of $V$ have highest weights equal to $\mu_s:=\mu_0+s\beta$ for $s=0,1,\cdots$, where
\begin{equation*}
\mu_0 = (0,\cdots,0;m-n,0,\cdots,0), \quad \quad \beta = (1,0,\cdots,0;1,0,\cdots,0).
\end{equation*}

\begin{defi} \label{def-mn}
Let ${\bf v}=(v_1,\cdots,v_m,v_{m+1},\cdots, v_{m+n})$ be an $(m+n)$-vector. We define the following:
\begin{enumerate}
\item ${\bf v}(m) := (v_1,\cdots,v_m)$ is called the {\bf $m$-part of ${\bf v}$}
\item ${\bf v}(n) :=(v_{m+1},\cdots,v_n)$ is called the {\bf $n$-part of ${\bf v}$}.
\item ${\bf v}^\#:=(|v_1|,\cdots,|v_{m+n}|)$ is called the {\bf absolute of ${\bf v}$}.
\end{enumerate}
\end{defi}

\begin{defi}
Let $N=\{m+n-3,\cdots,1,0\}$, and $P\subseteq N$ be such that 
$$P:=\{a_1,\cdots,a_{m-1}\}, \quad \quad N-P :=\{b_1,\cdots,b_{n-1}\},$$ 
where $a_1\geq\cdots\geq a_{m-1}$ and $b_1\geq\cdots\geq b_{n-1}$. For $\epsilon,\sigma\in\{\pm 1\}$, define:
\begin{align*}
\lambda_P^{\epsilon} &:=
\begin{cases}
(a_1+1,\cdots,a_{m-1}+1,0;\ b_1+1,\cdots,b_{n-1}+1,\epsilon) & \text{if\ $a_{m-1}=0$, i.e. $0\in P$} \\
(a_1+1,\cdots,a_{m-1}+1,\epsilon;\ b_1+1,\cdots,b_{n-1}+1,0) & \text{if\ $b_{n-1}=0$, i.e. $0\notin P$} 
\end{cases}    \\
\rho_P^{\epsilon,\sigma} &:=
\begin{cases}
(m+n-\frac{3-\sigma}{2},a_1,\cdots,a_{m-2},a_{m-1};m+n-\frac{3+\sigma}{2},b_1,\cdots,b_{n-2},\epsilon b_{n-1}) & \text{if $a_{m-1}=0$} \\
(m+n-\frac{3-\sigma}{2},a_1,\cdots,a_{m-2},\epsilon a_{m-1};m+n-\frac{3+\sigma}{2},b_1,\cdots,b_{n-2},b_{n-1}) & \text{if $b_{n-1}=0$} 
\end{cases}    
\end{align*}
(note that we may write $\pm$ instead of $\pm 1$ by abuse of notations).
\end{defi}

Here is the main result of this section:
\begin{prop}\label{prop-sopqeven}
The set of all Dirac triples of $V$ is the union of:
\begin{equation*}
\{(\mu_n,\rho_P^{-\epsilon,+},\lambda_P^\epsilon),~(\mu_{n-1},\rho_P^{\epsilon,-},\lambda_P^\epsilon)|~0\in P\subseteq N,~|P|=m-1,~\epsilon\in\{\pm 1\}\}
\end{equation*}
and
\begin{equation*}
\{(\mu_n,\rho_P^{-\epsilon,-},\lambda_P^\epsilon),~(\mu_{n-1},\rho_P^{\epsilon,+},\lambda_P^\epsilon)|~0\notin P\subseteq N,~|P|=m-1,~\epsilon\in\{\pm 1\}\}.
\end{equation*}
\end{prop}

The proof of Proposition \ref{prop-sopqeven} consists of a couple of lemmas. We begin with a lemma giving necessary conditions on possible $\mu$, $\rho$, $\lambda$ separately.

\begin{lemma}\label{lemma-sopqeven-para}
Suppose $(\mu_s,\rho,\lambda)$ is a Dirac triple of $V$. Then the following conditions hold:
\begin{enumerate}
\item $\lambda=\lambda_P^\epsilon$ for some $P$, $\epsilon$;
\item $\rho=\rho_P^{\epsilon,\sigma}$ for some $P$, $\epsilon$ and $\sigma$;
\item $s=n$ or $s=n-1$.
\end{enumerate}
\end{lemma}

\begin{proof}
~
\begin{enumerate}
\item If $(\mu_s,\rho,\lambda)$ is a Dirac triple, then $\lambda-\rho_\frk$ has to be K-dominant, which implies that both $\lambda^\#(m)$ and $\lambda^\#(n)$ are strictly decreasing sequences.
Therefore, $1$ occur at most once in both $\lambda^\#(m)$ and $\lambda^\#(n)$.
More precisely, one of $\lambda(m)$ and $\lambda(n)$ ends with $(\cdots,1,0)$, and the other ends with $(\cdots,\pm 1)$.
Hence every $\lambda\in\Lambda_K(V)$ such that $\lambda-\rho_\frk$ is K-dominant is of the form $\lambda_P^\epsilon$.

\item Let $\rho=(x_1,\cdots,x_m;y_1,\cdots,y_n)$.
It suffices to prove that $\{x_1,y_1\}=\{m+n-1,m+n-2\}$.
Notice that both $\rho^\#(m)$ and $\rho^\#(n)$ are decreasing sequences, 
so we need only to show that $m+n-1$ and $m+n-2$ do not occur in the same part of $\rho^\#$.

Suppose on contrary that both of them occur in $\rho^\#(m)$ (the proof for both occurring in $\rho^\#(n)$ is similar), then $|x_2|=m+n-2$ (here we write $\left|x_2\right|$ instead of $x_2$ to include the case when $m=n=2$) and thus $(m+n-2)-(m-2)=n$ must occur in $(\mu_s-\rho+\rho_\frk)^\#(m)$.

However, $n$ cannot occur in $(\mu_s-\rho+\rho_\frk)^\#(m)=(\lambda-\rho_\frk)^\#(m)$. Indeed, one can easily check that for any choice of $\lambda \in \Lambda_K(V)$,
the greatest term in $(\lambda-\rho_\frk)(m)$ as well as $(\lambda-\rho_\frk)^\#(m)$ is at most $(m+n-2)-(m-1)=n-1$. 
So $m+n-1$ and $m+n-2$ cannot occur in $\rho(m)$ simultaneously.

\item 
According to (2), we can assume that $\rho=\rho_P^{\epsilon,\sigma}$ and $\lambda=\lambda_{P'}^{\epsilon'}$. Since $\{\mu_s-\rho+\rho_\frk\}=\lambda-\rho_\frk$, this implies that $(\mu_s-\rho+\rho_\frk)^\#(m)$ and $(\lambda-\rho_\frk)^\#(m)$ differ by a permutation
and so do $(\mu_s-\rho+\rho_\frk)^\#(n)$ and $(\lambda-\rho_\frk)^\#(n)$.
Thus, the sum of coordinates of $(\mu_s-\rho+\rho_\frk)^\#$ and that of $(\lambda-\rho_\frk)^\#$ are equal. In other words,
\begin{align*}
&\left|s-(m+n-\frac{3-\sigma}{2})+(m-1)\right|+\left|s-(m+n-\frac{3+\sigma}{2})+(n-1)+(m-n)\right| \\
&+\frac{(m+n-3)(m+n-2)}{2}-\frac{(m-2)(m-1)}{2}-\frac{(n-2)(n-1)}{2}
\end{align*}
is equal to
\begin{align*}
\frac{(m+n-1)(m+n)}{2}-\frac{(m-1)m}{2}-\frac{(n-1)n}{2}-(m+n-2).
\end{align*}
This equation can be reduced to $|n-s|+|(n-1)-s|=1$.
So either $s=n$ or $s=n-1$.
\end{enumerate}
\end{proof}

We have narrowed down the possible choices of $\rho$ and $\lambda$. The following lemma gives a further refinement of their possibilities.
\begin{lemma}\label{lemma-sopqeven-P}
Suppose $(\mu_s,\rho_{P}^{\epsilon,\sigma},\lambda_{P'}^{\epsilon’})$ is a Dirac triple of $V$, then $P=P'$.
\end{lemma}

\begin{proof}
Let 
\begin{align*}
&P=\{a_1,\cdots,a_{m-1}\}, &N-P =\{b_1,\cdots,b_{n-1}\} \\
&P'=\{a_1',\cdots,a_{m-1}'\}, &N-P'=\{b_1',\cdots,b_{n-1}'\}.
\end{align*}

We begin by studying $\{\mu_s-\rho_{P}^{\epsilon,\sigma}+\rho_\frk\}$. Notice that $\mu_s-\rho_{P}^{\epsilon,\sigma}+\rho_\frk$ is equal to:
\begin{equation} \label{eq-mneven}
\begin{aligned}
\Big(&s-\left(n-\frac{1-\sigma}{2}\right),(m-2)-a_1,\cdots,1-a_{m-2},-\epsilon a_{m-1};\\
&s-\left(n-\frac{1+\sigma}{2}\right),(n-2)-b_1,\cdots,1-b_{n-2},-\epsilon b_{n-1}\Big).
\end{aligned}
\end{equation}

By Lemma \ref{lemma-sopqeven-para}(3), $s = n$ or $n-1$. Actually, for whatever $s\in\{n,n-1\}$ and $\sigma\in\{\pm1\}$,
\begin{equation*}
\{\text{the first coordinate of }(\mu_s-\rho_{P}^{\epsilon,\sigma}+\rho_\frk)^\#(m),\text{ the first coordinate of }(\mu_s-\rho_{P}^{\epsilon,\sigma}+\rho_\frk)^\#(n)\}
\end{equation*}
is equal to
\begin{align*}
\left\{\left|s-(n-\frac{1-\sigma}{2})\right|,\left|s-(n-\frac{1+\sigma}{2})\right|\right\} = \{|s-(n-1)|,|s-n|\} = \{0,1\}.
\end{align*}

Now we focus on the last coordinates of $(\mu_s-\rho_{P}^{\epsilon,\sigma}+\rho_\frk)^\#(m)$ and $(\mu_s-\rho_{P}^{\epsilon,\sigma}+\rho_\frk)^\#(n)$. We {\it claim} that
\begin{equation*}
a_{m-1}\geq\left|s-(n-\frac{1-\sigma}{2})\right|,\quad \quad b_{n-1}\geq\left|s-(n-\frac{1+\sigma}{2})\right|.
\end{equation*}
In fact, one of the two inequalities above must hold because one of $|s-(n-\frac{1-\sigma}{2})|$ and $|s-(n-\frac{1+\sigma}{2})|$ is $0$. We only write down the proof for the case that $|s-(n-\frac{1-\sigma}{2})|=0$, and the proof is similar when $|s-(n-\frac{1+\sigma}{2})|=0$.

Suppose $|s-(n-\frac{1-\sigma}{2})|=0$, then $a_{m-1}$ must be $0$, otherwise $b_{n-1}=0$ so that $0$ will occur at both the $m$-part and the $n$-part of $\{\mu_s-\rho_{P}^{\epsilon,\sigma}+\rho_\frk\}=\lambda_{P'}^{\epsilon'}-\rho_\frk$. 
However, $0$ cannot occur in both parts of $\lambda_{P'}^{\epsilon’}-\rho_\frk$ by the definition of $\lambda_{P'}^{\epsilon'}$. So $a_{m-1}=0$, which means that $b_{n-1}\geq1 = |s-(n-\frac{1+\sigma}{2})|$ and the claim holds. 

As a consequence of the above claim, $\{\mu_s-\rho_{P}^{\epsilon,\sigma}+\rho_\frk\}$ is equal to
\begin{align}\label{eq-sopqeven-left}
&\Big(a_1-(m-2),\cdots,a_{m-2}-1,a_{m-1},\epsilon(s-(n-\frac{1-\sigma}{2}));\\
\nonumber&b_1-(m-2),\cdots,b_{n-2}-1,b_{n-1},\epsilon(s-(n-\frac{1+\sigma}{2}))\Big),
\end{align}
and hence $\{\mu_s-\rho_{P}^{\epsilon,\sigma}+\rho_\frk\}^\#$ is equal to
\begin{align}\label{eq-sopqeven-left-abs}
&\Big(a_1-(m-2),\cdots,a_{m-2}-1,a_{m-1},\left|s-(n-\frac{1-\sigma}{2})\right|;\\
\nonumber&b_1-(m-2),\cdots,b_{n-2}-1,b_{n-1},\left|s-(n-\frac{1+\sigma}{2})\right|\Big),
\end{align}

On the other hand, note that $\lambda_{P'}^{\epsilon’}-\rho_\frk$ is given by:
\begin{equation}\label{eq-sopqeven-right}
\begin{cases}
(a_1'-(m-2),\cdots,a_{m-1}',\epsilon';b_1'-(n-2),\cdots,b_{n-1}',0), &\text{if $0\in P$}; \\
(a_1'-(m-2),\cdots,a_{m-1}',0;b_1'-(n-2),\cdots,b_{n-1}',\epsilon'), &\text{if $0\notin P$}, \\
\end{cases}
\end{equation}
so $(\lambda_{P'}^{\epsilon’}-\rho_\frk)^\#$ is equal to
\begin{equation}\label{eq-sopqeven-right-abs}
\begin{cases}
(a_1'-(m-2),\cdots,a_{m-1}',1;b_1'-(n-2),\cdots,b_{n-1}',0), &\text{if $0\in P$}; \\
(a_1'-(m-2),\cdots,a_{m-1}',0;b_1'-(n-2),\cdots,b_{n-1}',1), &\text{if $0\notin P$}. \\
\end{cases}
\end{equation}
By hypothesis we have $\{\mu_s-\rho_{P}^{\epsilon,\sigma}+\rho_\frk\}^\#=(\lambda_{P'}^{\epsilon’}-\rho_\frk)^\#$, so (\ref{eq-sopqeven-left-abs}) and (\ref{eq-sopqeven-right-abs}) implies $a_i=a_i'$ for all $i=1,\dots,m-1$. In other words, $P=P'$.
\end{proof}

\begin{proof}[Proof of Proposition \ref{prop-sopqeven}]

By Lemma \ref{lemma-sopqeven-para} and Lemma \ref{lemma-sopqeven-P}, a Dirac triple $(\mu,\rho,\lambda)$ of $V$ is of the form $(\mu_s,\rho_P^{\epsilon,\sigma},\lambda_P^{\epsilon'})$ where $s\in\{n,n-1\}$.
To prove Proposition \ref{prop-sopqeven}, we need only to find out the relation among $\epsilon$, $\epsilon'$ and $\sigma$.

Putting (\ref{eq-sopqeven-left}) and (\ref{eq-sopqeven-right}) into $\{\mu_s-\rho_{P}^{\epsilon,\sigma}+\rho_\frk\}=\lambda_{P}^{\epsilon'}-\rho_\frk$, one observes that:
\begin{itemize}

\item If $0\in P$, then 
$\left(\epsilon(s-(n-\frac{1-\sigma}{2})),\epsilon(s-(n-\frac{1+\sigma}{2}))\right)=(\epsilon',0)$,
which implies
\begin{equation*}
\begin{cases}
s=n-1,~\epsilon\epsilon'=-1 &\text{when } \sigma=+;\\
s=n,~\epsilon\epsilon'=1 &\text{when } \sigma=-.
\end{cases}
\end{equation*}

\item If $0\notin P$, then 
$\left(\epsilon(s-(n-\frac{1-\sigma}{2})), \epsilon(s-(n-\frac{1+\sigma}{2}))\right) = (0,\epsilon')$, which implies
\begin{equation*}
\begin{cases}
s=n,~\epsilon\epsilon'=1 &\text{when } \sigma=+;\\
s=n-1,~\epsilon\epsilon'=-1 &\text{when } \sigma=-.
\end{cases}
\end{equation*}

\end{itemize}
In conclusion, one is left with the following candidates for Dirac triples:
\begin{equation*}
\begin{cases}
(\mu_{n-1},\rho_P^{-\epsilon,+},\lambda_P^\epsilon),~(\mu_n,\rho_P^{\epsilon,-},\lambda_P^\epsilon)&\text{where $0\in P$}; \\
(\mu_n,\rho_P^{\epsilon,+},\lambda_P^\epsilon),~(\mu_{n-1},\rho_P^{-\epsilon,-},\lambda_P^\epsilon)&\text{where $0\notin P$}.
\end{cases}
\end{equation*}
Moreover, this proof and the proof of Lemma \ref{lemma-sopqeven-P} have actually verified that they are indeed Dirac triples. Hence they are exactly all of the Dirac triples of $V$.
\end{proof}

By (\ref{eq-dc-min-repn}) and (\ref{eq-di-min-repn}), we can use the data of Dirac triples listed in Proposition \ref{prop-sopqeven} to obtain the following:
\begin{cor}\label{cor-sopqeven}
The Dirac cohomology and the Dirac index of $V$ are:
\begin{equation*}
H_D(V)=\bigoplus_{P,\epsilon}2\wt{E}_{\lambda_P^\epsilon-\rho_\frk}; \quad \quad \DI(V)=0.
\end{equation*}
\end{cor}

\section{$\mf{g}_0=\mf{so}(2m+1,2n+1)$ $(m\geq2,~n\geq1,~m\geq n)$}\label{sec:sopqodd}

Let $\frg_0=\mf{so}(2m+1,2n+1)$ where $m\geq2,~n\geq1,~m\geq n$. We choose compatible positive root systems $\Phi^+(\frg,\frh)$ and $\Phi^+(\frk,\frt)$ as following:
\renewcommand\arraystretch{1.2}
\begin{center}
\begin{tabular}{c|c}\hline
$\mf{g}$ & $\mf{so}(2m+2n+2,\mb{C})$ \\ 
$\mf{k}$ & $\mf{so}(2m+1,\mb{C})\oplus\mf{so}(2n+1,\mb{C})$ \\ 
$\Pi(\mf{g},\mf{h})$ &  $\{e_i-e_{i+1}|i=1,\cdots,m+n\}\cup\{e_{m+n}+e_{m+n+1}\}$ \\
$\Pi(\mf{k},\mf{t})$ &  $\{e_i-e_{i+1}|i,j=1,\cdots,m+n-1;~i\neq m\}\cup\{e_m,e_{m+n}\}$ \\
$\rho_\frg\in R_K(\mf{g})$ & $(m+n,m+n-1,\cdots,3,2,1,0)$ \\ 
$\rho_\frk$ & $(m-\hf,m-1-\hf,\cdots,2-\hf,1-\hf;n-\hf,n-1-\hf,\cdots,2-\hf,1-\hf;0)$ \\ \hline
\end{tabular}
\end{center}
\renewcommand\arraystretch{1}

There is only one minimal $\left(\frg,K\right)$-module $V:=V_{min}(\mf{so}\left(2m+1,2n+1\right))$, whose infinitesimal character is
\begin{equation*}
\lambda_0=\left(m+n-1,m+n-2,\cdots,2,1,1,0\right)\in\Lambda_K\left(V\right).
\end{equation*}
The $K$-types of $V$ have highest weights $\mu_s:=\mu_0+s\beta$ for $s=0,1,\cdots$, where
\begin{equation*}
\mu_0=(0,\cdots,0;m-n,0,\cdots,0;0), \quad \quad \beta=(1,0,\cdots,0;1,0,\cdots,0;0).
\end{equation*}
(Here we use semicolons just after the $m$-th coordinate and the $m+n$-th coordinate.) Note that $rank(G) \neq rank(K)$ in this case, so the Dirac index of $V$ is not defined.

\begin{rmk}
If $(\mu_s,\rho,\lambda)$ is a Dirac triple, then the last coordinate of $\lambda$ is 0, because the last coordinate of both $\{\mu_s-\rho+\rho_\frk\}$ (it is extended from a weight in $\Phi(\mf{k},\mf{t})$) and $\rho_\frk$ must be 0.
In this case, we omit the last coordinate and deal with vector ${\bf v}=(v_1,\cdots,v_m,v_{m+1},\cdots, v_{m+n})$ of length $m+n$ from now on, and define
${\bf v}(m)$, ${\bf v}(n)$, ${\bf v}^\#$ as in Definition \ref{def-mn}.
\end{rmk}

\begin{defi}
Let $N=\{m+n-2,\cdots,2,1\}$, and $P\subseteq N$ be such that 
$$P:=\{a_1,\cdots,a_{m-1}\}, \quad \quad N-P :=\{b_1,\cdots,b_{n-1}\},$$ 
where $a_1\geq\cdots\geq a_{m-1}$ and $b_1\geq\cdots\geq b_{n-1}$. For $\sigma\in\{\pm 1\}$, define:
\begin{align*}
&\lambda_P:=(a_1+1,\cdots,a_{m-1}+1,1;b_1+1,\cdots,b_{n-1}+1,1); \\
&\rho_P^\sigma:=(m+n-\frac{1-\sigma}{2},a_1,\cdots,a_{m-1};m+n-\frac{1+\sigma}{2},b_1,\cdots,b_{n-1}).  
\end{align*}
\end{defi}

\medskip
\begin{prop}\label{prop-sopqodd}
The set of all Dirac triples of $V$ is:
\begin{equation*}
\{(\mu_n,\rho_P^\sigma,\lambda_P)|~P\subseteq N,~|P|=m-1,~\sigma\in\{\pm 1\}\}.
\end{equation*}
\end{prop}

\medskip
\begin{lemma}\label{lemma-sopqodd-para}
Suppose $(\mu_s,\rho,\lambda)$ is a Dirac triple of $V$, then the following condition hold:
\begin{enumerate}
\item $\lambda=\lambda_P$ for some $P$;
\item $\rho=\rho_P^\sigma$ for some $P$ and $\sigma$;
\item $s=n$.
\end{enumerate}
\end{lemma}

\begin{proof}
~
\begin{enumerate}
\item 
We need only to prove that both the last coordinate of $\lambda(m)$ and that of $\lambda(n)$ are 1.
Obviously, there must be exactly two 1's occurring in the coordinates of $\lambda$, because there are two 1's occurring in $\lambda_0$.
Moreover, the two 1's cannot be in the same part because $\lambda-\rho_\frk$ has to be K-dominant.

\item It suffices to prove that $m+n$ and $m+n-1$ do not occur in the same part of $\rho$.
We show this by contradiction.
Suppose both $m+n$ and $m+n-1$ occur in $\rho(m)$, then the second term of $\rho(m)$ is $m+n-1$, thus $(m+n-1)-(m-1-\hf)=n+\hf$ occurs in $\{\mu_s-\rho+\rho_\frk\}(m)=(\lambda-\rho_\frk)(m)$.
However, $(\lambda-\rho_\frk)(m)$ is a non-increasing sequence, hence the greatest term (which is also the first term) in $(\lambda-\rho_\frk)(m)$ is less than or equal to $(m+n-1)-(m-\hf)=n-\hf$. 
This results in a contradiction, so $m+n$ and $m+n-1$ cannot occur in $\rho(m)$ simultaneously. The case is similar for $\rho(n)$.
Therefore, $m+n$ and $m+n-1$ must be in different parts of $\rho$, then $\rho$ is of the form $\rho_P^\sigma$ for some $P$ and $\sigma$.

\item 
Notice that the action of $W(\mf{k},\mf{t})\cong(S_m\ltimes(\mb{Z}_2)^m)\times(S_n\ltimes(\mb{Z}_2)^n)$ is by compositions of permutations and changing signs. Meanwhile, the action of $\{\cdot\}$ can be done by acting an element of $W(\mf{k},\mf{t})$.
So $\{\mu_s-\rho+\rho_\frk\}=\lambda-\rho_\frk$ implies that $(\mu_s-\rho+\rho_\frk)^\#$ and $(\lambda-\rho_\frk)^\#=\lambda-\rho_\frk$ are conjugated by a permutation.
Therefore, the sum of coordinates of $(\mu_s-\rho+\rho_\frk)^\#$ is equal to that of $\lambda-\rho_\frk$.
Namely,
\begin{align*}
&\left|s-(m+n-\frac{1-\sigma}{2})+(m-\hf)\right|+\left|s-(m+n-\frac{1+\sigma}{2})+(n-\hf)+(m-n)\right| \\
&+\frac{(m+n-1)(m+n-2)}{2}-\frac{(m-1)^2}{2}-\frac{(n-1)^2}{2}
\end{align*}
is equal to 
\begin{equation*}
\frac{(m+n+1)(m+n)}{2}-\left(m+n-1\right)-\frac{(m-1)^2}{2}-\left(m-\hf\right)-\frac{(n-1)^2}{2}-\left(n-\hf\right),
\end{equation*}
This equation can be simplified to
\begin{equation}\label{eq-sopqodd-s}
\left|s-n-\frac{\sigma}{2}\right|+\left|s-n+\frac{\sigma}{2}\right|=1.
\end{equation}
To satisfy (\ref{eq-sopqodd-s}), $s$ must be $n$.

\end{enumerate}
\end{proof}

\medskip
\begin{lemma}\label{lemma-sopqodd-P}
Suppose $(\mu_n,\rho_{P'}^\sigma,\lambda_P)$ is a Dirac triple of $V$, then $P=P'$.
\end{lemma}

\begin{proof}
Assume $P=\{a_1,\cdots,a_{m-1}\}$ and $N-P=\{b_1,\cdots,b_{n-1}\}$ where $a_1\geq\cdots\geq a_{m-1}$ and $b_1\geq\cdots\geq b_{n-1}$. 
Similarly, assume $P'=\{a_1',\cdots,a_{m-1}'\}$ and $N-P'=\{b_1',\cdots,b_{n-1}'\}$ where $a_1'\geq\cdots\geq a_{m-1}'$ and $b_1'\geq\cdots\geq b_{n-1}'$.

Notice that $\mu_s-\rho_{P'}^\sigma+\rho_\frk$ is
\begin{align}\label{eq-sopqodd-P-1}
&\Big(s-(m+n-\frac{1-\sigma}{2})+(m-\hf),(m-1-\hf)-a_1',\cdots,(1-\hf)-a_{m-1}';\\
\nonumber&s-(m+n-\frac{1+\sigma}{2})+(n-\hf)+(m-n),(n-1-\hf)-b_1',\cdots,(1-\hf)-b_{n-1}'\Big).
\end{align}
Recall that $s=n$, we have
\begin{align}\label{eq-sopqodd-P-2}
&\left\{\left|(s-(m+n-\frac{1-\sigma}{2})+(m-\hf)\right|,\left|s-(m+n-\frac{1+\sigma}{2})+(n-\hf)+(m-n)\right|\right\} \\
\nonumber=&\left\{\left|s-n-\frac{\sigma}{2}\right|,\left|s-n+\frac{\sigma}{2}\right|\right\}=\left\{\hf\right\}.
\end{align}
On the other hand, both $a_{m-1}'$ and $b_{n-1}'$ are positive integers, so
\begin{equation}\label{eq-sopqodd-P-3}
a_{m-1}'-\hf\geq\hf, \quad \quad b_{n-1}'-\hf\geq\hf.
\end{equation}
Meanwhile, by our hypothesis
\begin{equation}\label{eq-sopqodd-P-4}
a_{i}'-a_{i+1}'\geq1, \quad \quad b_{j}'-b_{j+1}'\geq1
\end{equation}
for $i=1,\cdots,m-2$ and $j=1,\cdots,n-2$.
By (\ref{eq-sopqodd-P-1}) (\ref{eq-sopqodd-P-2}) (\ref{eq-sopqodd-P-3}) (\ref{eq-sopqodd-P-4}), $\{\mu_s-\rho_{P'}^\sigma+\rho_\frk\}$ is given by
\begin{equation}\label{eq-sopqodd-P-5}
\Big(a_1'-(m-1-\hf),\cdots,a_{m-1}'-(1-\hf),\hf;~b_1'-(n-1-\hf),\cdots,b_{n-1}'-(1-\hf),\hf\Big).
\end{equation}

On the other hand, $\lambda_P-\rho_\frk$ is equal to
\begin{equation}\label{eq-sopqodd-P-6}
\Big(a_1+1-(m-\hf),\cdots,a_{m-1}+1-(2-\hf),\hf;b_1+1-(n-\hf),\cdots,b_{n-1}+1-(2-\hf),\hf\Big)
\end{equation}

Then $\{\mu_s-\rho_{P'}^\sigma+\rho_\frk\}=\lambda_P-\rho_\frk$ implies $a_i=a_i'$, $i=1,\cdots,m-1$, i,e, $P=P'$.
\end{proof}

\medskip
\begin{proof}[Proof of Proposition \ref{prop-sopqodd}]
By Lemma \ref{lemma-sopqodd-para} and Lemma \ref{lemma-sopqodd-P}, a Dirac triple of $V$ must be of the form $(\mu_n,\rho_P^\sigma,\lambda_P)$

It has been guaranteed as a hypothesis in the proof of Lemma \ref{lemma-sopqodd-P} that $\{\mu_s-\rho_P^\sigma+\rho_\frk\}=\lambda_P-\rho_\frk$ holds for every $P$ and $\sigma$, hence every $(\mu_n,\rho_P^\sigma,\lambda_P)$ is indeed a Dirac triple of $V$.
\end{proof}

\medskip
\begin{cor}\label{cor-sopqodd}
The Dirac cohomology of $V$ is
\begin{equation*}
H_D(V)=\bigoplus_{P}2\wt{E}_{\lambda_P-\rho_\frk}.
\end{equation*}
\end{cor}

\section{$\frg_0=\mf{so}(2n,3)$ $(n\geq2)$}\label{sec:so2n3}

Let $\frg_0=\mf{so}(2n,3)$ where $(n\geq2)$. Compatible positive root systems $\Phi^+(\frg,\frh)$ and $\Phi^+(\frk,\frt)$ are chosen as following:

\renewcommand\arraystretch{1.2}
\begin{center}
\begin{tabular}{c|c}\hline
$\mf{g}$ & $\mf{so}(2n+3,\mb{C})$ \\ 
$\mf{k}$ & $\mf{so}(2n,\mb{C})\oplus\mf{so}(3,\mb{C})$ \\ 
$\Pi(\mf{g},\mf{h})$ &  $\{e_i-e_{i+1}|i=1,\cdots,n\}\cup\{e_{n+1}\}$ \\
$\Pi(\mf{k},\mf{t})$ &  $\{e_i-e_{i+1}|i=1,\cdots,n-1\}\cup\{e_{n-1}+e_n,~e_{n+1}\}$ \\
$\rho_\frg\in R_K(\mf{g})$ & $(n+\frac{1}{2},n-\frac{1}{2},\cdots ,\frac{5}{2},\frac{3}{2},\frac{1}{2})$ \\ 
$\rho_\frk$ & $(n-1,n-2,\cdots ,1,0,\frac{1}{2})$ \\ \hline
\end{tabular}
\end{center}
\renewcommand\arraystretch{1}

There is only one minimal $\left(\frg,K\right)$-module $V:=V_{min}\left(\mf{so}\left(2n,3\right)\right)$, whose infinitesimal character is
\begin{equation*}
\lambda_0=\left(n-\frac{1}{2},n-\frac{3}{2},\cdots ,\frac{3}{2},\frac{1}{2},1\right)\in\Lambda_K\left(V\right).
\end{equation*}
We didn't choose the $\Phi^+(\mf{g},\mf{h})$-dominant one to be $\lambda_0$, but it doesn't matter. We do this to give a better parameterization of $\Lambda_K(V)$.

The $K$-types of $V$ have highest weights $\mu_s:=\mu_0+s\beta$ for $s=0,1,\cdots$, where
\begin{equation*}
\mu_0=\left(0,0,\cdots ,0,0,n-\frac{3}{2}\right), \quad \quad \beta=\left(1,0,\cdots ,0,0,1\right).
\end{equation*}

Notice that $W(\mf{g},\mf{h})\cong S_{n+1}\ltimes(\mb{Z}_2)^{n+1}$ and $W(\mf{k},\mf{t})\cong(S_n\ltimes(\mb{Z}_2)^{n-1})\times\mb{Z}_2$. A vector in $\frt^*$ (regarded as a sequence of numbers) is $W(\mf{k},\mf{t})$-dominant if its first $n$ terms are non-increasing, and the last term is non-negative.

For $\epsilon\in\{\pm 1\}$, let
\begin{equation}
\rho_i^\epsilon:=
\begin{cases}
(n+\frac{3}{2}-1,\cdots,\widehat{n+\frac{3}{2}-i},\cdots, \frac{3}{2},\frac{\epsilon}{2},n+\frac{3}{2}-i) & \text{ if $i=1,\cdots,n$;} \\
(n+\frac{1}{2},n-\frac{1}{2},\cdots ,\frac{5}{2},\frac{3\epsilon}{2},\frac{1}{2})& \text{ if $i=n+1$,} \\
\end{cases}
\end{equation}
where $\widehat{\bullet}$ means the $\bullet$ term is omitted. Also, let
\begin{equation}
\lambda_i^\epsilon:=
\begin{cases}
(n+\frac{1}{2}-1,\cdots ,\widehat{n+\frac{1}{2}-i},\cdots , \frac{3}{2},1,\frac{\epsilon}{2},n+\frac{1}{2}-i) & \text{ if $i=1,\cdots,n-1$;} \\
(n-\frac{1}{2},n-\frac{3}{2},\cdots ,\frac{3}{2},\epsilon,\frac{1}{2}) & \text{ if $i=n$;} \\
(n-\frac{1}{2},n-\frac{3}{2},\cdots ,\frac{3}{2},\frac{\epsilon}{2},1) & \text{ if $i=n+1$.} \\
\end{cases}
\end{equation}
Then $R_K(\mf{g})=\{\rho_i^\epsilon|\ i=1,\cdots,n+1;\ \epsilon= \pm 1\}$ 
and $\Lambda_K(V)=\{\lambda_i^\epsilon|\ i=1,\cdots,n+1;\epsilon=\pm 1\}$.
Define ${\bf v}^\#$ as in Definition \ref{def-mn}.

\begin{prop}\label{prop-so2n3}
If $n$ is even, then the Dirac triples of $V$ are
\begin{equation*}
    (\mu_1,\rho_1^+,\lambda_{n+1}^-),~(\mu_1,\rho_1^-,\lambda_{n+1}^+),~(\mu_1,\rho_2^+,\lambda_{n+1}^+),~(\mu_1,\rho_2^-,\lambda_{n+1}^-).
\end{equation*}
If $n$ is odd, then the Dirac triples of $V$ are
\begin{equation*}
    (\mu_1,\rho_1^+,\lambda_{n+1}^+),~(\mu_1,\rho_1^-,\lambda_{n+1}^-),~(\mu_1,\rho_2^+,\lambda_{n+1}^-),~(\mu_1,\rho_2^-,\lambda_{n+1}^+).
\end{equation*}
\end{prop}

\ms

\begin{lemma}\label{lemma-so2n3-1}
If $(\mu_s,\rho_{i_1}^{\epsilon_1},\lambda_{i_2}^{\epsilon_2})$ is Dirac, then $i_2=n+1$.
\end{lemma}

\begin{proof}
For $\epsilon\in\{\pm 1\}$, we have
\begin{equation}\label{eq-so2n3-left}
\mu_s-\rho_{i_1}^{\epsilon_1}+\rho_\frk:=
\begin{cases}
-(\frac{1}{2}-s,\frac{1}{2},\cdots ,\frac{1}{2},\frac{\epsilon_1}{2},\frac{3}{2}-s) & \text{ if $i_1=1$;} \\
-({\scriptsize \overbrace{\frac{3}{2}-s,\frac{3}{2},\cdots ,\frac{3}{2}}^{i_1-1\ terms}},\frac{1}{2},\cdots ,\frac{1}{2},\frac{\epsilon_1}{2},\frac{5}{2}-i_1-s) & \text{ if $i_1=2,\cdots,n-1$;}\\
-(\frac{3}{2}-s,\frac{3}{2},\cdots ,\frac{3}{2},\frac{3\epsilon_1}{2},\frac{3}{2}-n-s) & \text{ if $i_1=n+1$.} 
\end{cases}
\end{equation}
and 
\begin{equation}\label{eq-so2n3-right}
\lambda_{i_2}^{\epsilon_2}-\rho_\frk:=
\begin{cases}
({\scriptsize \overbrace{\frac{1}{2},\cdots ,\frac{1}{2}}^{i_2-1\ terms}},-\frac{1}{2},\cdots ,-\frac{1}{2},0,\frac{\epsilon_2}{2},n-i) & \text{ if $i_2=1,\cdots,n-1$;} \\
(\frac{1}{2},\cdots ,\frac{1}{2},\epsilon_2,0) & \text{ if $i_2=n$;} \\
(\frac{1}{2},\cdots ,\frac{1}{2},\frac{\epsilon_2}{2},\frac{1}{2}) & \text{ if $i_2=n+1$.} \\
\end{cases}
\end{equation}
Notice that the action of $\{\cdot\}$ is just a composition of permutations and changing signs, so $(\mu_s,\rho_{i_1}^{\epsilon_1},\lambda_{i_2}^{\epsilon_2})$ is Dirac means that $(\mu_s-\rho_{i_1}^{\epsilon_1}+\rho_\frk)^\#$ and $(\lambda_{i_2}^{\epsilon_2}-\rho_\frk)^\#$ are conjugated by a permutation.

By (\ref{eq-so2n3-left}) every coordinate of $(\mu_s-\rho_{i_1}^{\epsilon_1}+\rho_\frk)^\#$ is not an integer for any $i_1$ and any $\epsilon_1$.
However, by (\ref{eq-so2n3-right}) there is an integer coordinate in $(\lambda_{i_2}^{\epsilon_2}-\rho_\frk)^\#$ unless $i_2=n+1$. 
So $i_2$ must be $n+1$ if $(\mu_s,\rho_{i_1}^{\epsilon_1},\lambda_{i_2}^{\epsilon_2})$ is Dirac.
\end{proof}

\begin{lemma}\label{lemma-so2n3-2}
If $(\mu_s,\rho_{i_1}^{\epsilon_1},\lambda_{n+1}^{\epsilon_2})$ is Dirac for some $i_1\in\{1,\cdots,n\}$, 
then it must be one of the following cases:
\begin{align*}
i_1=1, s=1; \\
i_1=2, s=1.
\end{align*}
\end{lemma}

\begin{proof}
According to (\ref{eq-so2n3-right}), $(\lambda_{n+1}^{\epsilon_2}-\rho_\frk)^\#=(\hf,\cdots,\hf)$. Meanwhile, by (\ref{eq-so2n3-left}) the second coordinate of $\mu_s-\rho_{i_1}^{\epsilon_1}+\rho_\frk$ is $-\frac{3}{2}$ if $i_1\geq3$.
Suppose $(\mu_s,\rho_{i_1}^{\epsilon_1},\lambda_{n+1}^{\epsilon_2})$ is Dirac, then $i_1$ can only be either $1$ or $2$.
Given $i_1\in\{1,2\}$, we also need the last coordinate of $(\mu_s-\rho_{i_1}^{\epsilon_1}+\rho_\frk)^\#$ and that of $(\lambda_{n+1}^{\epsilon_2}-\rho_\frk)^\#$ to be equal. More precisely, $|\frac{5}{2}-i_1-s|=\frac{1}{2}$ should be satisfied. Now the remaining possibilities are:
\begin{align*}
i_1=1, s=1; \\
i_1=1, s=2; \\
i_1=2, s=0; \\
i_1=2, s=1.
\end{align*}

The first coordinate of $\mu_2-\rho_1^{\epsilon_1}+\rho_\frk$ is $\frac{3}{2}$ while that of $\mu_0-\rho_2^{\epsilon_1}+\rho_\frk$ is $-\frac{3}{2}$, so both $(\mu_2,\rho_1^{\epsilon_1},\lambda_{n+1}^{\epsilon_2})$ and $(\mu_0,\rho_2^{\epsilon_1},\lambda_{n+1}^{\epsilon_2})$ are not Dirac.

So $(\mu_s,\rho_{i_1}^{\epsilon_1},\lambda_{n+1}^{\epsilon_2})$ is Dirac implies either $i_1=1,s=1$ or $i_1=2,s=1$.

\end{proof}

\begin{proof}[Proof of Proposition \ref{prop-so2n3}]
By Lemma \ref{lemma-so2n3-1} and Lemma \ref{lemma-so2n3-2}, there are only eight candidates for Dirac triples:
\begin{equation*}
(\mu_1,\rho_1^\pm,\lambda_{n+1}^\pm),~(\mu_1,\rho_2^\pm,\lambda_{n+1}^\pm).
\end{equation*}
We use (\ref{eq-so2n3-left}) and (\ref{eq-so2n3-right}) to check which of them are Dirac.

If $n$ is even, Dirac triples are:
\begin{equation}
(\mu_1,\rho_1^+,\lambda_{n+1}^-),~(\mu_1,\rho_1^-,\lambda_{n+1}^+),~(\mu_1,\rho_2^+,\lambda_{n+1}^+),~(\mu_1,\rho_2^-,\lambda_{n+1}^-).
\end{equation}
If $n$ is odd, Dirac triples are:
\begin{equation}
(\mu_1,\rho_1^+,\lambda_{n+1}^+),~(\mu_1,\rho_1^-,\lambda_{n+1}^-),~(\mu_1,\rho_2^+,\lambda_{n+1}^-),~(\mu_1,\rho_2^-,\lambda_{n+1}^+).
\end{equation}

\end{proof}

\medskip
\begin{cor}\label{cor-so2n3}
The Dirac cohomology and Dirac index of $V$ are given by:
\begin{equation}
H_D(V)=2\wt{E}_{\lambda_{n+1}^+-\rho_\frk}\oplus2\wt{E}_{\lambda_{n+1}^--\rho_\frk}
\end{equation}
\begin{equation}
\DI(V)=
\begin{cases}
    2\wt{E}_{\lambda_{n+1}^--\rho_\frk}-2\wt{E}_{\lambda_{n+1}^+-\rho_\frk} & \text{if $n$ is even} \\
    2\wt{E}_{\lambda_{n+1}^+-\rho_\frk}-2\wt{E}_{\lambda_{n+1}^--\rho_\frk} & \text{if $n$ is odd}. \\
\end{cases}
\end{equation}
\end{cor}

\begin{proof}
The length of $w$ such that $\rho=w\rho_\frg$ are listed below.

\begin{table}[!h]
\begin{minipage}{0.49\linewidth}
\centering
\begin{tabular}{c|l|l|c} 
$s$ & $\rho$ & $\lambda$ & length of $w$ \\ \hline
$1$ & $\rho_1^-$ & $\lambda_{n+1}^+$ & $n+1$ \\ 
$1$ & $\rho_2^+$ & $\lambda_{n+1}^+$ & $n-1$ \\ 
$1$ & $\rho_1^+$ & $\lambda_{n+1}^-$ & $n$ \\ 
$1$ & $\rho_2^-$ & $\lambda_{n+1}^-$ & $n$ \\ 
\end{tabular}
\caption{When $n$ is even}
\end{minipage}
\begin{minipage}{0.49\linewidth}
\centering
\begin{tabular}{c|l|l|c} 
$s$ & $\rho$ & $\lambda$ & length of $w$ \\ \hline
$1$ & $\rho_1^+$ & $\lambda_{n+1}^+$ & $n$ \\ 
$1$ & $\rho_2^-$ & $\lambda_{n+1}^+$ & $n$ \\ 
$1$ & $\rho_1^-$ & $\lambda_{n+1}^-$ & $n+1$ \\ 
$1$ & $\rho_2^+$ & $\lambda_{n+1}^-$ & $n-1$ \\ 
\end{tabular}
\caption{When $n$ is odd}
\end{minipage}
\end{table}
\end{proof}

\section{$\frg_0=\mf{f}_{4\left(4\right)}$}\label{sec:splitf4}

Let $\frg_0=\mf{f}_{4\left(4\right)}$. Compatible positive root systems $\Phi^+(\frg,\frh)$ and $\Phi^+(\frk,\frt)$ are chosen as following:

\renewcommand\arraystretch{1.2}
\begin{center}
\begin{tabular}{c|c}\hline
$\mf{g}$ & $\mf{f}_4(\mb{C})$ \\ 
$\mf{k}$ & $\mf{sp}(3,\mb{C})\oplus\mf{sl}(2,\mb{C})$ \\ 
$\Pi(\mf{g},\mf{h})$ &  $\{\hf(e_1-e_2-e_3-e_4),e_4,e_3-e_4,e_2-e_3\}$ \\
$\Pi(\mf{k},\mf{t})$ & $\{\hf(e_1-e_2-e_3-e_4),e_4,e_3-e_4,e_1+e_2\}$ \\
$\rho_\frg\in R_K(\mf{g})$ & $(\frac{11}{2},\frac{5}{2},\frac{3}{2},\frac{1}{2})$ \\ 
$\rho_\frk$ & $(2,-1,\frac{3}{2},\frac{1}{2})$ \\ \hline
\end{tabular}
\end{center}
\renewcommand\arraystretch{1}

By Table \ref{tab:nr-minimal-repns}, there is only one minimal $(\mf{g},K)$-module $V:=V_{min}(\mf{f}_{4\left(4\right)})$ with infinitesimal character $$\lambda_0 = \left(4,\frac{3}{2},1,\frac{1}{2}\right) \in \Lambda_K(V).$$ 
Also, the $K$-types of $V$ have highest weights equal to $\mu_s:=\mu_0+s\beta$ for $s=0,1,\cdots$, where
\begin{equation*}
\mu_0 = \left(\frac{1}{2},\frac{1}{2},0,0\right), \quad \beta = \left(1,0,1,0\right).
\end{equation*}

To simplify the computation, we will use another \underline{\textbf{orthogonal}} basis $f:=\{f_1,f_2,f_3,f_4\}$ of $\frh^*=\frt^*$, where
\begin{equation}
f_1:=\left(\frac{1}{2},\frac{1}{2},0,0\right), \quad f_2:=\left(\frac{1}{2},-\frac{1}{2},0,0\right), \quad f_3:=\left(0,0,\frac{1}{2},\frac{1}{2}\right), \quad f_4:\left(0,0,\frac{1}{2},-\frac{1}{2}\right).
\end{equation}
We write $v=(v_1,v_2,v_3,v_4)_f$ if $v=v_1f_1+v_2f_2+v_3f_3+v_4f_4$. Thus
\begin{equation}
e_1:=(1,1,0,0)_f, \quad e_2:=(1,-1,0,0)_f, \quad e_3:=(0,0,1,1)_f, \quad e_4:=(0,0,1,-1)_f.
\end{equation}
Now the root systems and important vectors can be rewritten in the new basis:
\renewcommand\arraystretch{1.2}
\begin{center}
\begin{tabular}{c|c}\hline
$\Pi(\mf{g},\mf{h})$ & $\{f_2-f_3,f_3-f_4,2f_4,f_1-f_2-f_3-f_4\}$ \\
$\Pi(\mf{k},\mf{t})$ & $\{f_2-f_3,f_3-f_4,2f_4,2f_1\}$ \\
$\rho_\frg\in R_K(\mf{g})$ & $(8,3,2,1)_f$ \\ 
$\rho_\frk$ & $(1,3,2,1)_f$ \\ 
$\lambda_0\in\Lambda_K(V)$ & $\left(\frac{11}{2},\frac{5}{2},\frac{3}{2},\frac{1}{2}\right)_f$ \\
$\mu_0$ & $(1,0,0,0)_f$ \\
$\beta$ & $(1,1,1,1)_f$ \\ \hline
\end{tabular}
\end{center}
\renewcommand\arraystretch{1}

\ms

\begin{prop}\label{prop-f44}

All Dirac triples $(\mu_s,\rho,\lambda)$ of $V$ are listed in the following table with the length of $w\in W\left(\frg,\frt\right)$ such that $\rho=w\rho_\frg$:
\begin{center}
\begin{tabular}{c|l|l|c} 
$s$ & $\rho$ & $\lambda$ & length of $w$ \\ \hline
$2$ & $(6,5,4,1)_f$ & $(3,5,2,1)_f$ & $12$ \\ 
$2$ & $(2,7,4,3)_f$ & $(3,5,2,1)_f$ & $8$ \\ 
$2$ & $(4,6,5,1)_f$ & $(1,5,3,2)_f$ & $10$ \\ 
$2$ & $(4,7,3,2)_f$ & $(1,5,3,2)_f$ & $10$ \\ 
$2$ & $(5,6,4,1)_f$ & $(2,5,3,1)_f$ & $11$ \\ 
$2$ & $(3,7,4,2)_f$ & $(2,5,3,1)_f$ & $9$ \\ 
$1$ & $(7,4,3,2)_f$ & $(5,3,2,1)_f$ & $13$ \\ 
$3$ & $(1,6,5,4)_f$ & $(5,3,2,1)_f$ & $7$ \\ 
\end{tabular}
\end{center}
By means of (\ref{eq-dc}) and (\ref{eq-di}), the Dirac cohomology and the Dirac index of $V$ can be computed explicitly:
\begin{align}
H_D(V)=2\wt{E}_{(3,5,2,1)_f-\rho_\frk}+2\wt{E}_{(1,5,3,2)_f-\rho_\frk}+2\wt{E}_{(2,5,3,1)_f-\rho_\frk}+2\wt{E}_{(5,3,2,1)_f-\rho_\frk}; \\
\DI(V)=2\wt{E}_{(3,5,2,1)_f-\rho_\frk}+2\wt{E}_{(1,5,3,2)_f-\rho_\frk}-2\wt{E}_{(2,5,3,1)_f-\rho_\frk}-2\wt{E}_{(5,3,2,1)_f-\rho_\frk}.
\end{align}
\end{prop}

\ms

For $\alpha\in \Phi\left(\frg,\frh\right)$, denote by $s_\alpha\in W\left(\frg,\frh\right)$ the reflection with respect to $\alpha$. Let
\begin{equation}
\Phi'=\{\pm 2f_i\pm 2f_j|i,j=1,2,3,4,i\neq j\}\cup\{\pm 2f_i|i=1,2,3,4\}.
\end{equation}
Set $W(\Phi')$ to be the subgroup of $W\left(\frg,\frh\right)$ generated by $\{s_\alpha|\alpha\in\Phi'\}$. Obviously, $W(\Phi')\cong S_4\ltimes(\mb{Z}_2)^4$. More precisely, $W(\Phi')$ acts on $\frh^*=\frt^*$ by compositions of permutations and changing signs.

\begin{lemma}\label{lemma-f44-weyl}
$W(\mf{g},\mf{h})$ has exactly three right $W(\Phi')$-cosets, they are 
$$W(\Phi'), \quad W(\Phi')s_{\frac{1}{2}(f_1+f_2+f_3+f_4)}, \quad W(\Phi')s_{\frac{1}{2}(-f_1+f_2+f_3+f_4)}.$$
In particular, the $W(\mf{g},\mf{h})$-orbit of $v\in\frh^*$ is the union of 3 $W(\Phi')$-orbits
$$W(\Phi')v, \quad W(\Phi')s_{\frac{1}{2}(f_1+f_2+f_3+f_4)}v, \quad W(\Phi')s_{\frac{1}{2}(-f_1+f_2+f_3+f_4)}v.$$
\end{lemma}

\begin{proof}
The quotient $\frac{\left|W(\frg,\frh)\right|}{\left|W(\Phi')\right|}$ of cardinalities is 3, so there are exactly three right $W(\Phi')$-cosets.
To prove Lemma \ref{lemma-f44-weyl}, it suffices to show that any two of
$$1, \quad s_{\frac{1}{2}(f_1+f_2+f_3+f_4)}, \quad s_{\frac{1}{2}(-f_1+f_2+f_3+f_4)}$$ are in different right $W(\Phi')$-cosets.
Apparently, 
\begin{align*}
W(\frg,\frh) &\to R(\frg) \\
w &\mapsto w\rho_\frg 
\end{align*}
gives a bijection.
If $w_1$ and $w_2$ are in the same right $W(\Phi')$-coset, $w_1\rho_\frg$ and $w_2\rho_\frg$ must be in the same $W(\Phi')$-orbit.
Recall that $W(\Phi')$ acts on $\frh^*=\frt^*$ by compositions of permutations and changing signs, so it's easy to verify whether two vectors in $\frh^*$ are in the same $W(\Phi')$-orbit.
Any two of 
$$\rho_\frg, \quad s_{\frac{1}{2}(f_1+f_2+f_3+f_4)}\rho_\frg, \quad s_{\frac{1}{2}(-f_1+f_2+f_3+f_4)}\rho_\frg$$
are not in the same $W(\Phi')$-orbit, so
$$W(\Phi'), \quad W(\Phi')s_{\frac{1}{2}(f_1+f_2+f_3+f_4)}, \quad W(\Phi')s_{\frac{1}{2}(-f_1+f_2+f_3+f_4)}$$ 
are exactly the three right $W(\Phi')$-cosets of $W(\frg,\frh)$.
\end{proof}

Lemma \ref{lemma-f44-weyl} is useful for describing the $W(\mf{g},\mf{h})$-orbit of any element in $\mf{h}^*$. Given $v\in\frh^*$, the $W(\frg,\frh)$-orbit of $v$ is
$$W(\Phi')v\cup W(\Phi')s_{\frac{1}{2}(f_1+f_2+f_3+f_4)}v\cup W(\Phi')s_{\frac{1}{2}(-f_1+f_2+f_3+f_4)}v.$$

The new basis makes it very easy to determine whether a member of $\mf{h}^*$ is K-dominant. Actually, 
$$(v_1,v_2,v_3,v_4)_f \text{ is K-dominant} \Longleftrightarrow v_1\geq0,~v_2\geq v_3\geq v_4\geq0.$$
Now we have an very effective way to list the elements of $R_K(\mf{g})$ and $\Lambda_K(V)$.
For instance, we list all $\rho \in R_K(\mf{g})$:
\begin{center}
$
\begin{array}{llll}
(\hspace{0.8em}8,\hspace{0.8em}3,\hspace{0.8em}2,\hspace{0.8em}1)_f & (\hspace{0.8em}3,\hspace{0.8em}8,\hspace{0.8em}2,\hspace{0.8em}1)_f & (\hspace{0.8em}2,\hspace{0.8em}8,\hspace{0.8em}3,\hspace{0.8em}1)_f & (\hspace{0.8em}1,\hspace{0.8em}8,\hspace{0.8em}3,\hspace{0.8em}2)_f \\
(\hspace{0.8em}6,\hspace{0.8em}5,\hspace{0.8em}4,\hspace{0.8em}1)_f & (\hspace{0.8em}5,\hspace{0.8em}6,\hspace{0.8em}4,\hspace{0.8em}1)_f & (\hspace{0.8em}4,\hspace{0.8em}6,\hspace{0.8em}5,\hspace{0.8em}1)_f & (\hspace{0.8em}1,\hspace{0.8em}6,\hspace{0.8em}5,\hspace{0.8em}4)_f \\
(\hspace{0.8em}7,\hspace{0.8em}4,\hspace{0.8em}3,\hspace{0.8em}2)_f & (\hspace{0.8em}4,\hspace{0.8em}7,\hspace{0.8em}3,\hspace{0.8em}2)_f & (\hspace{0.8em}3,\hspace{0.8em}7,\hspace{0.8em}4,\hspace{0.8em}2)_f & (\hspace{0.8em}2,\hspace{0.8em}7,\hspace{0.8em}4,\hspace{0.8em}3)_f,
\end{array}
$
\end{center}
and the corresponding $\mu_0+\rho_\frk-\rho$ are:
\begin{center}
$
\begin{array}{llll}
(-6,\hspace{0.8em}0,\hspace{0.8em}0,\hspace{0.8em}0)_f & (-1,-5,\hspace{0.8em}0,\hspace{0.8em}0)_f & (\hspace{0.8em}0,-5,-1,\hspace{0.8em}0)_f & (\hspace{0.8em}1,-5,-1,-1)_f \\
(-4,-2,-2,\hspace{0.8em}0)_f & (-3,-3,-2,\hspace{0.8em}0)_f & (-2,-3,-3,\hspace{0.8em}0)_f & (\hspace{0.8em}1,-3,-3,-3)_f \\
(-5,-1,-1,-1)_f & (-2,-4,-1,-1)_f & (-1,-4,-2,-1)_f & (\hspace{0.8em}0,-4,-2,-2)_f.
\end{array}
$
\end{center}

$\Lambda_K(V)$ consists of
\begin{center}
$
\begin{array}{llll}
\hf(\hspace{0.4em}11,\hspace{0.8em}5,\hspace{0.8em}3,\hspace{0.8em}1)_f & \hf(\hspace{0.8em}5,\hspace{0.4em}11,\hspace{0.8em}3,\hspace{0.8em}1)_f & \hf(\hspace{0.8em}3,\hspace{0.4em}11,\hspace{0.8em}5,\hspace{0.8em}1)_f & \hf(\hspace{0.8em}1,\hspace{0.4em}11,\hspace{0.8em}5,\hspace{0.8em}3)_f \\
\hf(\hspace{0.8em}9,\hspace{0.8em}7,\hspace{0.8em}5,\hspace{0.8em}1)_f & \hf(\hspace{0.8em}7,\hspace{0.8em}9,\hspace{0.8em}5,\hspace{0.8em}1)_f & \hf(\hspace{0.8em}5,\hspace{0.8em}9,\hspace{0.8em}7,\hspace{0.8em}1)_f & \hf(\hspace{0.8em}1,\hspace{0.8em}9,\hspace{0.8em}7,\hspace{0.8em}5)_f \\
\hspace{0.6em}(\hspace{0.8em}5,\hspace{0.8em}3,\hspace{0.8em}2,\hspace{0.8em}1)_f & \hspace{0.6em}(\hspace{0.8em}3,\hspace{0.8em}5,\hspace{0.8em}2,\hspace{0.8em}1)_f & \hspace{0.6em}(\hspace{0.8em}2,\hspace{0.8em}5,\hspace{0.8em}3,\hspace{0.8em}1)_f & \hspace{0.6em}(\hspace{0.8em}1,\hspace{0.8em}5,\hspace{0.8em}3,\hspace{0.8em}2)_f.
\end{array}
$
\end{center}
When looking for Dirac triples $(\mu,\rho,\lambda)$, we need only to consider $\lambda \in \Lambda_K(V)$ such that $\lambda-\rho_\frk$ is K-dominant.
Then there are only 4 possible choices for $\lambda$:
\begin{equation}
(5,3,2,1)_f,~(3,5,2,1)_f,~(2,5,3,1)_f,~(1,5,3,2)_f,
\end{equation}
whose corresponding $\lambda-\rho_\frk$ are
\begin{equation}\label{eq-f44-lambda-rhoc}
(4,0,0,0)_f,~(2,2,0,0)_f,~(1,2,1,0)_f,~(0,2,1,1)_f.
\end{equation}

\ms

Assume $\mu_0+\rho_\frk-\rho=(x_1,x_2,x_3,x_4)_f$.
To reduce the remaining computations, we should make use of the fact that $\beta=(1,1,1,1)_f$.

When $\lambda-\rho_\frk=(4,0,0,0)_f$, there must be $x_2+s=x_3+s=x_4+s=0$, hence $x_2=x_3=x_4$. Therefore $\mu_0+\rho_\frk-\rho$ is one of
\begin{center}
$
\begin{array}{lll}
    (-6,\hspace{0.8em}0,\hspace{0.8em}0,\hspace{0.8em}0)_f, & (\hspace{0.8em}1,-3,-3,-3)_f, & (-5,-1,-1,-1)_f.
\end{array}
$
\end{center}
For each $\mu_0+\rho_\frk-\rho$ listed above, we need only to consider at most two $s$ such that $\left|x_1+s\right|=4$. Among the 5 possibilities, there are 2 Dirac triples:
\begin{center}
\begin{tabular}{c|l|l|c} 
$s$ & $\mu_s$ & $\rho$ & $\lambda$ \\ \hline
$1$ & $(2,1,1,1)_f$ & $(7,4,3,2)_f$ & $(5,3,2,1)_f$ \\ 
$3$ & $(4,3,3,3)_f$ & $(1,6,5,4)_f$ & $(5,3,2,1)_f$ \\ 
\end{tabular}
\end{center}

When $\lambda-\rho_\frk=(2,2,0,0)_f$, exactly two of $x_2+s$, $x_3+s$, $x_4+s$ are equal to 0, which means exactly two of $x_2$, $x_3$, $x_4$ are equal. In this case, $\mu_0+\rho_\frk-\rho$ must be one of
\begin{center}
$
\begin{array}{lll}
    (-1,-5,\hspace{0.8em}0,\hspace{0.8em}0)_f & (\hspace{0.8em}1,-5,-1,-1)_f & (-4,-2,-2,\hspace{0.8em}0)_f \\
    (-2,-3,-3,\hspace{0.8em}0)_f & (-2,-4,-1,-1)_f & (\hspace{0.8em}0,-4,-2,-2)_f.
\end{array}
$
\end{center}
For each $\mu_0+\rho_\frk-\rho$ listed above, we need only to consider at most two $s$ such that $\left|x_1+s\right|=2$. Among the 9 possibilities, there are 2 Dirac triples:
\begin{center}
\begin{tabular}{c|l|l|c} 
$s$ & $\mu_s$ & $\rho$ & $\lambda$ \\ \hline
$2$ & $(3,2,2,2)_f$ & $(6,5,4,1)_f$ & $(3,5,2,1)_f$ \\ 
$2$ & $(3,2,2,2)_f$ & $(2,7,4,3)_f$ & $(3,5,2,1)_f$ \\ 
\end{tabular}
\end{center}

When $\lambda-\rho_\frk=(1,2,1,0)_f$, $x_2+s$, $x_3+s$, $x_4+s$ must be mutually distinct, so the same is true for $x_2$, $x_3$, $x_4$. Therefore, $\mu_0+\rho_\frk-\rho$ can only be
\begin{center}
$
\begin{array}{lll}
    (\hspace{0.8em}0,-5,-1,\hspace{0.8em}0)_f & (-3,-3,-2,\hspace{0.8em}0)_f & (-1,-4,-2,-1)_f.
\end{array}
$
\end{center}
For each $\mu_0+\rho_\frk-\rho$ listed above, we need only to consider at most two $s$ such that $\left|x_1+s\right|=1$. Among the 5 possibilities, there are 2 Dirac triples:
\begin{center}
\begin{tabular}{c|l|l|c} 
$s$ & $\mu_s$ & $\rho$ & $\lambda$ \\ \hline
$2$ & $(3,2,2,2)_f$ & $(5,6,4,1)_f$ & $(2,5,3,1)_f$ \\ 
$2$ & $(3,2,2,2)_f$ & $(3,7,4,2)_f$ & $(2,5,3,1)_f$ \\ 
\end{tabular}
\end{center}

When $\lambda-\rho_\frk=(0,2,1,1)_f$, for each $\rho$, we need only to try the only $s$ such that $\left|x_1+s\right|=0$. Among the 12 possibilities, there are 2 Dirac triples:
\begin{center}
\begin{tabular}{c|l|l|c} 
$s$ & $\mu_s$ & $\rho$ & $\lambda$ \\ \hline
$2$ & $(3,2,2,2)_f$ & $(4,6,5,1)_f$ & $(1,5,3,2)_f$ \\ 
$2$ & $(3,2,2,2)_f$ & $(4,7,3,2)_f$ & $(1,5,3,2)_f$ \\ 
\end{tabular}
\end{center}

The above discussion exhausts all of Dirac triples of $V$, so we've completed the proof of Proposition \ref{prop-f44}.

\section{$\frg_0=\mf{e}_{8\left(8\right)}$}\label{sec:splite8}

Let $\frg_0=\mf{e}_{8\left(8\right)}$. Compatible positive root systems $\Phi^+(\frg,\frh)$ and $\Phi^+(\frk,\frt)$ are chosen as following:

\renewcommand\arraystretch{1.2}
\begin{center}
\begin{tabular}{c|c}\hline
$\mf{g}$ & $\mf{e}_8(\mb{C})$ \\ 
$\mf{k}$ & $\mf{so}(16,\mb{C})$ \\ 
$\Pi(\mf{g},\mf{h})$ & $\{e_{i+1}-e_i|i=1,\cdots,6\}\cup\{e_2+e_1,\hf(e_8-e_7-e_6-e_5-e_4-e_3-e_2+e_1)\}$ \\
$\Pi(\mf{k},\mf{t})$ & $\{e_{i+1}-e_i|i=1,\cdots,7\}\cup\{e_8+e_7\}$ \\
$\rho_\frg\in R_K(\mf{g})$ & $(0,1,2,3,4,5,6,23)$ \\ 
$\rho_\frk$ & $(0,1,2,3,4,5,6,7)$ \\ \hline
\end{tabular}
\end{center}
\renewcommand\arraystretch{1}

By Table \ref{tab:nr-minimal-repns}, there is only one minimal $(\mf{g},K)$-module $V:=V_{min}(\mf{e}_{8\left(8\right)})$ with infinitesimal character $$\lambda_0 = \left(0,1,1,2,3,4,5,18\right) \in \Lambda_K(V).$$ 
Also, the $K$-types of $V$ have highest weights equal to $\mu_s:=\mu_0+s\beta$ for $s=0,1,\cdots$, where
\begin{equation*}
\mu_0 = \left(0,0,0,0,0,0,0,0\right), \quad \beta = \left(\frac{1}{2},\frac{1}{2},\frac{1}{2},\frac{1}{2},\frac{1}{2},\frac{1}{2},\frac{1}{2},\frac{1}{2}\right).
\end{equation*}

In this case, we use \texttt{Python} to find all Dirac triples of $V$. The \texttt{Python} codes we used are attached in appendix.
Although there are infinitely many choices of $s$, 
the following proposition implies that there is an upper bound for $s$ in order for $(\mu_s,\rho,\lambda)$ to be a Dirac triple:
\begin{prop}\rm \label{prop-upper-bound-of-s}
Let $V$ be a minimal $(\frg,K)$-module with $K$-types whose highest weights are $\mu_s:=\mu_0+s\beta$ where $s=0,1,\cdots$.
If $(\mu_s,\rho,\lambda)$ is a Dirac triple of $V$, i.e.
\begin{equation}
\{\mu_0+s\beta-\rho+\rho_\frk\}=\lambda-\rho_\frk
\end{equation}
for some $\rho\in R_K(\mf{g})$ and $\lambda\in \Lambda_K(V)$,
then
\begin{equation}
s\leq\frac{\norm{\rho}+\norm{\lambda}}{\left\lVert \beta\right\rVert}
\end{equation}
\end{prop}

\begin{proof}
Suppose $(\mu_s,\rho,\lambda)$ is a Dirac triple of some representation and $s>\frac{\norm{\rho}+\norm{\lambda}}{\left\lVert \beta\right\rVert}$, then
\begin{align*}
    \norm{\{\mu_0+s\beta-\rho+\rho_\frk\}+\rho_\frk} &\geq \norm{\mu_0+s\beta-\rho+\rho_\frk}-\norm{\rho_\frk} \\
    &\geq \norm{\mu_0+s\beta}-\norm{\rho-\rho_\frk}-\norm{\rho_\frk} \\
    &\geq \norm{\mu_0+s\beta}-\norm{\rho} \\
    &\geq s\norm{\beta}-\norm{\rho} \\
    &>\norm{\lambda}
\end{align*}
Thus 
\begin{equation*}
\{\mu_0+s\beta-\rho+\rho_\frk\}=\lambda-\rho_\frk
\end{equation*}
cannot hold, by contradiction, $s\leq\frac{\norm{\rho}+\norm{\lambda}}{\left\lVert \beta\right\rVert}$.
\end{proof}

Notice that $\norm{\rho}$ and $\norm{\lambda}$ does not depend on $\rho$ and $\lambda$ once $R(\mf{g})$ and $\Lambda(V)$ are fixed. So, for a fixed minimal $(\frg,K)$-module $V$, $\frac{\norm{\rho}+\norm{\lambda}}{\left\lVert \beta\right\rVert}$ is a constant.
When $V:=V_{min}(\mf{e}_{8\left(8\right)})$, $\frac{\norm{\rho}+\norm{\lambda}}{\left\lVert \beta\right\rVert}<32$.

\ms

All Dirac triples $(\mu_s,\rho,\lambda)$ of $V$ are listed in the following table.
For each of them, we give the length of $w\in W(\mf{g},\mf{h})$ such that $\rho=w\rho_\frg$.
From this table we can see that every $\lambda\in\Lambda_K(V)$ appears in either zero or two Dirac triples. If $\lambda$ appears in two Dirac triples, then one of the two corresponds to a $\Kt$-type in $H_D(V)^+$ and the other corresponds to the same $\Kt$-type in $H_D(V)^-$, therefore $\DI(V)=0$.

\renewcommand\arraystretch{1.4}
\begin{longtable}{|c|l|l|c|} \hline
    \caption{Dirac triples of $V_{min}(\mf{e}_{8\left(8\right)})$} \\ \hline
    \endhead
    $s$ & $\rho$ & $\lambda$ & length of $w$ \\ \hline
    $ 7 $ & $ (-\frac{1}{2}, \frac{3}{2}, \frac{5}{2}, \frac{17}{2}, \frac{19}{2}, \frac{3}{2}, \frac{25}{2}, \frac{27}{2}) $ & $ (-2, 3, 4, 6, 7, 8, 9, 11) $ & $ 19 $ \\\hline
    $ 8 $ & $ (1, 2, 4, 5, 6, 12, 13, 15) $ & $ (-2, 3, 4, 6, 7, 8, 9, 11) $ & $ 16 $ \\\hline
    $ 7 $ & $ (-1, 2, 3, 8, 9, 11, 12, 14) $ & $ (-\frac{3}{2}, \frac{5}{2}, \frac{9}{2}, \frac{11}{2}, \frac{13}{2}, \frac{15}{2}, \frac{19}{2}, \frac{23}{2}) $ & $ 18 $ \\\hline
    $ 8 $ & $ (\frac{1}{2}, \frac{5}{2}, \frac{7}{2}, \frac{11}{2}, \frac{13}{2}, \frac{23}{2}, \frac{25}{2}, \frac{31}{2}) $ & $ (-\frac{3}{2}, \frac{5}{2}, \frac{9}{2}, \frac{11}{2}, \frac{13}{2}, \frac{15}{2}, \frac{19}{2}, \frac{23}{2}) $ & $ 15 $ \\\hline
    $ 7 $ & $ (1, 2, 3, 6, 7, 11, 12, 16) $ & $ (-\frac{1}{2}, \frac{3}{2}, \frac{9}{2}, \frac{11}{2}, \frac{13}{2}, \frac{15}{2}, \frac{17}{2}, \frac{25}{2}) $ & $ 14 $ \\\hline
    $ 8 $ & $ (-\frac{3}{2}, \frac{5}{2}, \frac{7}{2}, \frac{15}{2}, \frac{17}{2}, \frac{23}{2}, \frac{25}{2}, \frac{27}{2}) $ & $ (-\frac{1}{2}, \frac{3}{2}, \frac{9}{2}, \frac{11}{2}, \frac{13}{2}, \frac{15}{2}, \frac{17}{2}, \frac{25}{2}) $ & $ 19 $ \\\hline
    $ 7 $ & $ (-\frac{3}{2}, \frac{5}{2}, \frac{7}{2}, \frac{15}{2}, \frac{19}{2}, \frac{3}{2}, \frac{23}{2}, \frac{29}{2}) $ & $ (-1, 3, 4, 5, 6, 7, 10, 12) $ & $ 17 $ \\\hline
    $ 8 $ & $ (0, 3, 4, 5, 7, 11, 12, 16) $ & $ (-1, 3, 4, 5, 6, 7, 10, 12) $ & $ 14 $ \\\hline
    $ 7 $ & $ (\frac{1}{2}, \frac{5}{2}, \frac{7}{2}, \frac{11}{2}, \frac{15}{2}, \frac{3}{2}, \frac{23}{2}, \frac{33}{2}) $ & $ (0, 2, 4, 5, 6, 7, 9, 13) $ & $ 13 $ \\\hline
    $ 8 $ & $ (-2, 3, 4, 7, 9, 11, 12, 14) $ & $ (0, 2, 4, 5, 6, 7, 9, 13) $ & $ 18 $ \\\hline
    $ 7 $ & $ (1, 2, 4, 5, 8, 10, 11, 17) $ & $ (\frac{1}{2}, \frac{5}{2}, \frac{7}{2}, \frac{9}{2}, \frac{11}{2}, \frac{15}{2}, \frac{17}{2}, \frac{27}{2}) $ & $ 12 $ \\\hline
    $ 8 $ & $ (-\frac{5}{2}, \frac{7}{2}, \frac{9}{2}, \frac{13}{2}, \frac{19}{2}, \frac{3}{2}, \frac{25}{2}, \frac{27}{2}) $ & $ (\frac{1}{2}, \frac{5}{2}, \frac{7}{2}, \frac{9}{2}, \frac{11}{2}, \frac{15}{2}, \frac{17}{2}, \frac{27}{2}) $ & $ 19 $ \\\hline
    $ 7 $ & $ (-2, 3, 4, 8, 9, 10, 11, 15) $ & $ (-\frac{3}{2}, \frac{5}{2}, \frac{7}{2}, \frac{9}{2}, \frac{11}{2}, \frac{13}{2}, \frac{3}{2}, \frac{25}{2}) $ & $ 16 $ \\\hline
    $ 8 $ & $ (-\frac{1}{2}, \frac{7}{2}, \frac{9}{2}, \frac{11}{2}, \frac{13}{2}, \frac{3}{2}, \frac{23}{2}, \frac{33}{2}) $ & $ (-\frac{3}{2}, \frac{5}{2}, \frac{7}{2}, \frac{9}{2}, \frac{11}{2}, \frac{13}{2}, \frac{3}{2}, \frac{25}{2}) $ & $ 13 $ \\\hline
    $ 7 $ & $ (0, 3, 4, 6, 7, 10, 11, 17) $ & $ (-\frac{1}{2}, \frac{3}{2}, \frac{7}{2}, \frac{9}{2}, \frac{11}{2}, \frac{13}{2}, \frac{19}{2}, \frac{27}{2}) $ & $ 12 $ \\\hline
    $ 8 $ & $ (-\frac{5}{2}, \frac{7}{2}, \frac{9}{2}, \frac{15}{2}, \frac{17}{2}, \frac{3}{2}, \frac{23}{2}, \frac{29}{2}) $ & $ (-\frac{1}{2}, \frac{3}{2}, \frac{7}{2}, \frac{9}{2}, \frac{11}{2}, \frac{13}{2}, \frac{19}{2}, \frac{27}{2}) $ & $ 17 $ \\\hline
    $ 7 $ & $ (\frac{1}{2}, \frac{5}{2}, \frac{9}{2}, \frac{11}{2}, \frac{15}{2}, \frac{19}{2}, \frac{3}{2}, \frac{35}{2}) $ & $ (0, 2, 3, 4, 5, 7, 9, 14) $ & $ 11 $ \\\hline
    $ 8 $ & $ (-3, 4, 5, 7, 9, 10, 12, 14) $ & $ (0, 2, 3, 4, 5, 7, 9, 14) $ & $ 18 $ \\\hline
    $ 7 $ & $ (1, 2, 5, 6, 7, 9, 10, 18) $ & $ (-\frac{1}{2}, \frac{3}{2}, \frac{5}{2}, \frac{7}{2}, \frac{9}{2}, \frac{15}{2}, \frac{17}{2}, \frac{29}{2}) $ & $ 10 $ \\\hline
    $ 8 $ & $ (-\frac{7}{2}, \frac{9}{2}, \frac{11}{2}, \frac{15}{2}, \frac{17}{2}, \frac{19}{2}, \frac{25}{2}, \frac{27}{2}) $ & $ (-\frac{1}{2}, \frac{3}{2}, \frac{5}{2}, \frac{7}{2}, \frac{9}{2}, \frac{15}{2}, \frac{17}{2}, \frac{29}{2}) $ & $ 19 $ \\\hline
    $ 7 $ & $ (-\frac{5}{2}, \frac{9}{2}, \frac{11}{2}, \frac{13}{2}, \frac{15}{2}, \frac{17}{2}, \frac{27}{2}, \frac{29}{2}) $ & $ (0, 1, 2, 3, 4, 9, 10, 13) $ & $ 17 $ \\\hline
    $ 8 $ & $ (0, 1, 6, 7, 8, 9, 10, 17) $ & $ (0, 1, 2, 3, 4, 9, 10, 13) $ & $ 12 $ \\\hline
    $ 7 $ & $ (-2, 4, 5, 6, 8, 9, 13, 15) $ & $ (\frac{1}{2}, \frac{3}{2}, \frac{5}{2}, \frac{7}{2}, \frac{9}{2}, \frac{17}{2}, \frac{3}{2}, \frac{25}{2}) $ & $ 16 $ \\\hline
    $ 8 $ & $ (-\frac{1}{2}, \frac{3}{2}, \frac{11}{2}, \frac{13}{2}, \frac{17}{2}, \frac{19}{2}, \frac{3}{2}, \frac{33}{2}) $ & $ (\frac{1}{2}, \frac{3}{2}, \frac{5}{2}, \frac{7}{2}, \frac{9}{2}, \frac{17}{2}, \frac{3}{2}, \frac{25}{2}) $ & $ 13 $ \\\hline
    $ 7 $ & $ (-\frac{3}{2}, \frac{7}{2}, \frac{9}{2}, \frac{11}{2}, \frac{17}{2}, \frac{19}{2}, \frac{27}{2}, \frac{29}{2}) $ & $ (1, 2, 3, 4, 5, 9, 10, 12) $ & $ 17 $ \\\hline
    $ 8 $ & $ (0, 1, 5, 6, 9, 10, 11, 16) $ & $ (1, 2, 3, 4, 5, 9, 10, 12) $ & $ 14 $ \\\hline
    $ 7 $ & $ (-\frac{3}{2}, \frac{7}{2}, \frac{9}{2}, \frac{13}{2}, \frac{15}{2}, \frac{19}{2}, \frac{25}{2}, \frac{31}{2}) $ & $ (0, 1, 3, 4, 5, 8, 11, 12) $ & $ 15 $ \\\hline
    $ 8 $ & $ (-1, 2, 5, 7, 8, 10, 11, 16) $ & $ (0, 1, 3, 4, 5, 8, 11, 12) $ & $ 14 $ \\\hline
    $ 7 $ & $ (-1, 3, 4, 6, 8, 10, 13, 15) $ & $ (\frac{1}{2}, \frac{3}{2}, \frac{7}{2}, \frac{9}{2}, \frac{11}{2}, \frac{17}{2}, \frac{3}{2}, \frac{23}{2}) $ & $ 16 $ \\\hline
    $ 8 $ & $ (-\frac{1}{2}, \frac{3}{2}, \frac{9}{2}, \frac{13}{2}, \frac{17}{2}, \frac{3}{2}, \frac{23}{2}, \frac{31}{2}) $ & $ (\frac{1}{2}, \frac{3}{2}, \frac{7}{2}, \frac{9}{2}, \frac{11}{2}, \frac{17}{2}, \frac{3}{2}, \frac{23}{2}) $ & $ 15 $ \\\hline
    $ 7 $ & $ (-\frac{1}{2}, \frac{5}{2}, \frac{7}{2}, \frac{13}{2}, \frac{15}{2}, \frac{3}{2}, \frac{27}{2}, \frac{29}{2}) $ & $ (0, 1, 4, 5, 6, 9, 10, 11) $ & $ 17 $ \\\hline
    $ 8 $ & $ (0, 1, 4, 7, 8, 11, 12, 15) $ & $ (0, 1, 4, 5, 6, 9, 10, 11) $ & $ 16 $ \\\hline
    $ 7 $ & $ (-\frac{3}{2}, \frac{5}{2}, \frac{9}{2}, \frac{15}{2}, \frac{17}{2}, \frac{19}{2}, \frac{23}{2}, \frac{31}{2}) $ & $ (-1, 2, 3, 4, 6, 7, 11, 12) $ & $ 15 $ \\\hline
    $ 8 $ & $ (-1, 3, 5, 6, 7, 10, 12, 16) $ & $ (-1, 2, 3, 4, 6, 7, 11, 12) $ & $ 14 $ \\\hline
    $ 7 $ & $ (-1, 2, 4, 7, 9, 10, 12, 15) $ & $ (-\frac{1}{2}, \frac{5}{2}, \frac{7}{2}, \frac{9}{2}, \frac{13}{2}, \frac{15}{2}, \frac{3}{2}, \frac{23}{2}) $ & $ 16 $ \\\hline
    $ 8 $ & $ (-\frac{1}{2}, \frac{5}{2}, \frac{9}{2}, \frac{11}{2}, \frac{15}{2}, \frac{3}{2}, \frac{25}{2}, \frac{31}{2}) $ & $ (-\frac{1}{2}, \frac{5}{2}, \frac{7}{2}, \frac{9}{2}, \frac{13}{2}, \frac{15}{2}, \frac{3}{2}, \frac{23}{2}) $ & $ 15 $ \\\hline
    $ 7 $ & $ (-\frac{1}{2}, \frac{3}{2}, \frac{7}{2}, \frac{15}{2}, \frac{17}{2}, \frac{3}{2}, \frac{25}{2}, \frac{29}{2}) $ & $ (-1, 2, 4, 5, 7, 8, 10, 11) $ & $ 17 $ \\\hline
    $ 8 $ & $ (0, 2, 4, 6, 7, 11, 13, 15) $ & $ (-1, 2, 4, 5, 7, 8, 10, 11) $ & $ 16 $ \\\hline
    $ 7 $ & $ (0, 1, 3, 8, 9, 10, 13, 14) $ & $ (-\frac{3}{2}, \frac{5}{2}, \frac{7}{2}, \frac{11}{2}, \frac{15}{2}, \frac{17}{2}, \frac{19}{2}, \frac{3}{2}) $ & $ 18 $ \\\hline
    $ 8 $ & $ (\frac{1}{2}, \frac{3}{2}, \frac{9}{2}, \frac{11}{2}, \frac{13}{2}, \frac{23}{2}, \frac{27}{2}, \frac{29}{2}) $ & $ (-\frac{3}{2}, \frac{5}{2}, \frac{7}{2}, \frac{11}{2}, \frac{15}{2}, \frac{17}{2}, \frac{19}{2}, \frac{3}{2}) $ & $ 17 $ \\\hline
    $ 7 $ & $ (-\frac{3}{2}, \frac{5}{2}, \frac{11}{2}, \frac{13}{2}, \frac{15}{2}, \frac{3}{2}, \frac{23}{2}, \frac{31}{2}) $ & $ (0, 1, 2, 5, 6, 7, 11, 12) $ & $ 15 $ \\\hline
    $ 8 $ & $ (-1, 3, 4, 7, 8, 9, 12, 16) $ & $ (0, 1, 2, 5, 6, 7, 11, 12) $ & $ 14 $ \\\hline
    $ 7 $ & $ (-1, 2, 5, 6, 8, 11, 12, 15) $ & $ (\frac{1}{2}, \frac{3}{2}, \frac{5}{2}, \frac{11}{2}, \frac{13}{2}, \frac{15}{2}, \frac{3}{2}, \frac{23}{2}) $ & $ 16 $ \\\hline
    $ 8 $ & $ (-\frac{1}{2}, \frac{5}{2}, \frac{7}{2}, \frac{13}{2}, \frac{17}{2}, \frac{19}{2}, \frac{25}{2}, \frac{31}{2}) $ & $ (\frac{1}{2}, \frac{3}{2}, \frac{5}{2}, \frac{11}{2}, \frac{13}{2}, \frac{15}{2}, \frac{3}{2}, \frac{23}{2}) $ & $ 15 $ \\\hline
    $ 7 $ & $ (-\frac{1}{2}, \frac{3}{2}, \frac{9}{2}, \frac{13}{2}, \frac{15}{2}, \frac{23}{2}, \frac{25}{2}, \frac{29}{2}) $ & $ (0, 1, 3, 6, 7, 8, 10, 11) $ & $ 17 $ \\\hline
    $ 8 $ & $ (0, 2, 3, 7, 8, 10, 13, 15) $ & $ (0, 1, 3, 6, 7, 8, 10, 11) $ & $ 16 $ \\\hline
    $ 7 $ & $ (0, 1, 5, 6, 7, 12, 13, 14) $ & $ (-\frac{1}{2}, \frac{3}{2}, \frac{5}{2}, \frac{13}{2}, \frac{15}{2}, \frac{17}{2}, \frac{19}{2}, \frac{3}{2}) $ & $ 18 $ \\\hline
    $ 8 $ & $ (\frac{1}{2}, \frac{3}{2}, \frac{5}{2}, \frac{15}{2}, \frac{17}{2}, \frac{19}{2}, \frac{27}{2}, \frac{29}{2}) $ & $ (-\frac{1}{2}, \frac{3}{2}, \frac{5}{2}, \frac{13}{2}, \frac{15}{2}, \frac{17}{2}, \frac{19}{2}, \frac{3}{2}) $ & $ 17 $ \\\hline
    $ 7 $ & $ (\frac{1}{2}, \frac{3}{2}, \frac{11}{2}, \frac{13}{2}, \frac{15}{2}, \frac{17}{2}, \frac{3}{2}, \frac{35}{2}) $ & $ (0, 1, 2, 3, 5, 8, 9, 14) $ & $ 11 $ \\\hline
    $ 8 $ & $ (-3, 4, 6, 7, 8, 9, 13, 14) $ & $ (0, 1, 2, 3, 5, 8, 9, 14) $ & $ 18 $ \\\hline
    $ 7 $ & $ (0, 2, 5, 6, 8, 9, 11, 17) $ & $ (\frac{1}{2}, \frac{3}{2}, \frac{5}{2}, \frac{7}{2}, \frac{11}{2}, \frac{15}{2}, \frac{19}{2}, \frac{27}{2}) $ & $ 12 $ \\\hline
    $ 8 $ & $ (-\frac{5}{2}, \frac{7}{2}, \frac{11}{2}, \frac{13}{2}, \frac{17}{2}, \frac{19}{2}, \frac{25}{2}, \frac{29}{2}) $ & $ (\frac{1}{2}, \frac{3}{2}, \frac{5}{2}, \frac{7}{2}, \frac{11}{2}, \frac{15}{2}, \frac{19}{2}, \frac{27}{2}) $ & $ 17 $ \\\hline
    $ 7 $ & $ (\frac{1}{2}, \frac{3}{2}, \frac{9}{2}, \frac{11}{2}, \frac{17}{2}, \frac{19}{2}, \frac{23}{2}, \frac{33}{2}) $ & $ (1, 2, 3, 4, 6, 8, 9, 13) $ & $ 13 $ \\\hline
    $ 8 $ & $ (-2, 3, 5, 6, 9, 10, 13, 14) $ & $ (1, 2, 3, 4, 6, 8, 9, 13) $ & $ 18 $ \\\hline
    $ 7 $ & $ (-\frac{1}{2}, \frac{5}{2}, \frac{9}{2}, \frac{13}{2}, \frac{15}{2}, \frac{19}{2}, \frac{23}{2}, \frac{33}{2}) $ & $ (0, 1, 3, 4, 6, 7, 10, 13) $ & $ 13 $ \\\hline
    $ 8 $ & $ (-2, 3, 5, 7, 8, 10, 12, 15) $ & $ (0, 1, 3, 4, 6, 7, 10, 13) $ & $ 16 $ \\\hline
    $ 7 $ & $ (0, 2, 4, 6, 8, 10, 12, 16) $ & $ (\frac{1}{2}, \frac{3}{2}, \frac{7}{2}, \frac{9}{2}, \frac{13}{2}, \frac{15}{2}, \frac{19}{2}, \frac{25}{2}) $ & $ 14 $ \\\hline
    $ 8 $ & $ (-\frac{3}{2}, \frac{5}{2}, \frac{9}{2}, \frac{13}{2}, \frac{17}{2}, \frac{3}{2}, \frac{25}{2}, \frac{29}{2}) $ & $ (\frac{1}{2}, \frac{3}{2}, \frac{7}{2}, \frac{9}{2}, \frac{13}{2}, \frac{15}{2}, \frac{19}{2}, \frac{25}{2}) $ & $ 17 $ \\\hline
    $ 7 $ & $ (\frac{1}{2}, \frac{3}{2}, \frac{7}{2}, \frac{13}{2}, \frac{15}{2}, \frac{3}{2}, \frac{25}{2}, \frac{31}{2}) $ & $ (0, 1, 4, 5, 7, 8, 9, 12) $ & $ 15 $ \\\hline
    $ 8 $ & $ (-1, 2, 4, 7, 8, 11, 13, 14) $ & $ (0, 1, 4, 5, 7, 8, 9, 12) $ & $ 18 $ \\\hline
    $ 6 $ & $ (0, 1, 2, 9, 10, 11, 12, 13) $ & $ (-3, 4, 5, 6, 7, 8, 9, 10) $ & $ 20 $ \\\hline
    $ 9 $ & $ (\frac{3}{2}, \frac{5}{2}, \frac{7}{2}, \frac{9}{2}, \frac{11}{2}, \frac{25}{2}, \frac{27}{2}, \frac{29}{2}) $ & $ (-3, 4, 5, 6, 7, 8, 9, 10) $ & $ 17 $ \\\hline
    $ 6 $ & $ (\frac{3}{2}, \frac{5}{2}, \frac{7}{2}, \frac{9}{2}, \frac{17}{2}, \frac{19}{2}, \frac{3}{2}, \frac{35}{2}) $ & $ (\frac{3}{2}, \frac{5}{2}, \frac{7}{2}, \frac{9}{2}, \frac{11}{2}, \frac{13}{2}, \frac{15}{2}, \frac{29}{2}) $ & $ 11 $ \\\hline
    $ 9 $ & $ (-3, 4, 5, 6, 10, 11, 12, 13) $ & $ (\frac{3}{2}, \frac{5}{2}, \frac{7}{2}, \frac{9}{2}, \frac{11}{2}, \frac{13}{2}, \frac{15}{2}, \frac{29}{2}) $ & $ 20 $ \\\hline
    $ 6 $ & $ (\frac{3}{2}, \frac{5}{2}, \frac{9}{2}, \frac{11}{2}, \frac{15}{2}, \frac{17}{2}, \frac{19}{2}, \frac{37}{2}) $ & $ (\frac{1}{2}, \frac{3}{2}, \frac{5}{2}, \frac{7}{2}, \frac{9}{2}, \frac{13}{2}, \frac{15}{2}, \frac{31}{2}) $ & $ 9 $ \\\hline
    $ 9 $ & $ (-4, 5, 6, 7, 9, 10, 12, 13) $ & $ (\frac{1}{2}, \frac{3}{2}, \frac{5}{2}, \frac{7}{2}, \frac{9}{2}, \frac{13}{2}, \frac{15}{2}, \frac{31}{2}) $ & $ 20 $ \\\hline
    $ 6 $ & $ (2, 3, 4, 6, 7, 8, 9, 19) $ & $ (0, 1, 2, 3, 5, 6, 7, 16) $ & $ 8 $ \\\hline
    $ 9 $ & $ (-\frac{9}{2}, \frac{11}{2}, \frac{13}{2}, \frac{15}{2}, \frac{17}{2}, \frac{3}{2}, \frac{23}{2}, \frac{25}{2}) $ & $ (0, 1, 2, 3, 5, 6, 7, 16) $ & $ 21 $ \\\hline
    $ 6 $ & $ (1, 3, 4, 5, 8, 9, 10, 18) $ & $ (1, 2, 3, 4, 5, 6, 8, 15) $ & $ 10 $ \\\hline
    $ 9 $ & $ (-\frac{7}{2}, \frac{9}{2}, \frac{11}{2}, \frac{13}{2}, \frac{19}{2}, \frac{3}{2}, \frac{23}{2}, \frac{27}{2}) $ & $ (1, 2, 3, 4, 5, 6, 8, 15) $ & $ 19 $ \\\hline
    $ 5 $ & $ (\frac{5}{2}, \frac{7}{2}, \frac{9}{2}, \frac{11}{2}, \frac{13}{2}, \frac{15}{2}, \frac{17}{2}, \frac{39}{2}) $ & $ (0, 1, 2, 3, 4, 5, 6, 17) $ & $ 7 $ \\\hline
    $ 10 $ & $ (-5, 6, 7, 8, 9, 10, 11, 12) $ & $ (0, 1, 2, 3, 4, 5, 6, 17) $ & $ 22 $ \\\hline
\end{longtable}
\renewcommand\arraystretch{1}

\section{$\frg_0=\mf{e}_{8\left(-24\right)}$}\label{sec:quaternionice8}

Let $\frg_0=\mf{e}_{8\left(-24\right)}$. Compatible positive root systems $\Phi^+(\frg,\frh)$ and $\Phi^+(\frk,\frt)$ are chosen as following:

\renewcommand\arraystretch{1.2}
\begin{center}
\begin{tabular}{c|c}\hline
$\mf{g}$ & $\mf{e}_8(\mb{C})$ \\ 
$\mf{k}$ & $\mf{e}(7,\mb{C})\oplus\mf{sl}(2,\mb{C})$ \\ 
$\Pi(\mf{g},\mf{h})$ & $\{e_{i+1}-e_i|i=1,\cdots,6\}\cup\{e_2+e_1,\hf(e_8-e_7-e_6-e_5-e_4-e_3-e_2+e_1)\}$ \\
$\Pi(\mf{k},\mf{t})$ &  $\{e_{i+1}-e_i|i=1,\cdots ,5\}\cup\{e_1+e_2,e_7+e_8,\frac{1}{2}(e_1+e_8-\sum_{i=2}^7e_i)\}$ \\
$\rho_\frg\in R_K(\mf{g})$ & $(0,1,2,3,4,5,6,23)$ \\ 
$\rho_\frk$ & $(0,1,2,3,4,5,-8,9)$ \\  \hline
\end{tabular}
\end{center}
\renewcommand\arraystretch{1}

By Table \ref{tab:nr-minimal-repns}, there is only one minimal $(\mf{g},K)$-module $V:=V_{min}(\mf{e}_{8\left(-24\right)})$ with infinitesimal character $$\lambda_0 = \left(0,1,1,2,3,4,5,18\right) \in \Lambda_K(V).$$ 
Also, the $K$-types of $V$ have highest weights equal to $\mu_s:=\mu_0+s\beta$ for $s=0,1,\cdots$, where
\begin{equation*}
\mu_0 = \left(0,0,0,0,0,0,4,4\right), \quad \beta = \left(0,0,0,0,0,1,0,1\right).
\end{equation*}

In this case, we still use \texttt{Python} to find all Dirac triples of $V$.
When $V:=V_{min}(\mf{e}_{8\left(-24\right)})$, $\frac{\norm{\rho}+\norm{\lambda}}{\left\lVert \beta\right\rVert}<32$.

All Dirac triples of $V$ are listed below. In this case we still have $\DI(V)=0$.

\ms

\renewcommand\arraystretch{1.4}
\begin{longtable}{|c|l|l|c|} \hline
    \caption{Dirac triples of $V_{min}(\mf{e}_{8(-24)})$} \\ \hline
    \endhead
    $s$ & $\rho$ & $\lambda$ & length of $w$ \\ \hline
    $ 5 $ & $ (0, 1, 2, 3, 9, 10, -8, 19) $ & $ (0, 1, 2, 3, 4, 10, -9, 13) $ & $ 18 $ \\ \hline
    $ 6 $ & $ (0, 1, 2, 3, 4, 6, -5, 23) $ & $ (0, 1, 2, 3, 4, 10, -9, 13) $ & $ 11 $ \\ \hline
    $ 5 $ & $ (0, 1, 2, 4, 5, 6, -3, 23) $ & $ (\hf, \frac{3}{2}, \frac{5}{2}, \frac{7}{2}, \frac{9}{2}, \frac{17}{2}, -\frac{15}{2}, \frac{29}{2}) $ & $ 9 $ \\ \hline
    $ 6 $ & $ (-\hf, \frac{3}{2}, \frac{5}{2}, \frac{7}{2}, \frac{15}{2}, \frac{3}{2}, -\frac{19}{2}, \frac{37}{2}) $ & $ (\hf, \frac{3}{2}, \frac{5}{2}, \frac{7}{2}, \frac{9}{2}, \frac{17}{2}, -\frac{15}{2}, \frac{29}{2}) $ & $ 20 $ \\ \hline
    $ 5 $ & $ (0, 1, 2, 3, 5, 6, -4, 23) $ & $ (0, 1, 2, 3, 5, 9, -8, 14) $ & $ 10 $ \\ \hline
    $ 6 $ & $ (0, 1, 2, 3, 8, 10, -9, 19) $ & $ (0, 1, 2, 3, 5, 9, -8, 14) $ & $ 19 $ \\ \hline
    $ 5 $ & $ (0, 1, 3, 5, 6, 7, -4, 22) $ & $ (0, 2, 3, 4, 5, 7, -9, 14) $ & $ 11 $ \\ \hline
    $ 6 $ & $ (-\hf, \frac{3}{2}, \frac{7}{2}, \frac{9}{2}, \frac{13}{2}, \frac{19}{2}, -\frac{17}{2}, \frac{39}{2}) $ & $ (0, 2, 3, 4, 5, 7, -9, 14) $ & $ 18 $ \\ \hline
    $ 5 $ & $ (-\hf, \frac{3}{2}, \frac{5}{2}, \frac{11}{2}, \frac{13}{2}, \frac{15}{2}, -\frac{9}{2}, \frac{43}{2}) $ & $ (\hf, \frac{3}{2}, \frac{7}{2}, \frac{9}{2}, \frac{11}{2}, \frac{13}{2}, -\frac{19}{2}, \frac{27}{2}) $ & $ 12 $ \\ \hline
    $ 6 $ & $ (-1, 2, 3, 5, 6, 9, -8, 20) $ & $ (\hf, \frac{3}{2}, \frac{7}{2}, \frac{9}{2}, \frac{11}{2}, \frac{13}{2}, -\frac{19}{2}, \frac{27}{2}) $ & $ 17 $ \\ \hline
    $ 5 $ & $ (-\frac{3}{2}, \frac{5}{2}, \frac{7}{2}, \frac{9}{2}, \frac{11}{2}, \frac{17}{2}, -\frac{15}{2}, \frac{5}{2}) $ & $ (\frac{3}{2}, \frac{5}{2}, \frac{7}{2}, \frac{9}{2}, \frac{11}{2}, \frac{13}{2}, -\frac{3}{2}, \frac{25}{2}) $ & $ 16 $ \\ \hline
    $ 6 $ & $ (0, 1, 2, 6, 7, 8, -5, 21) $ & $ (\frac{3}{2}, \frac{5}{2}, \frac{7}{2}, \frac{9}{2}, \frac{11}{2}, \frac{13}{2}, -\frac{3}{2}, \frac{25}{2}) $ & $ 13 $ \\ \hline
    $ 5 $ & $ (0, 1, 3, 4, 6, 7, -5, 22) $ & $ (-\hf, \frac{3}{2}, \frac{5}{2}, \frac{7}{2}, \frac{11}{2}, \frac{15}{2}, -\frac{19}{2}, \frac{27}{2}) $ & $ 12 $ \\ \hline
    $ 6 $ & $ (0, 1, 3, 4, 7, 9, -8, 20) $ & $ (-\hf, \frac{3}{2}, \frac{5}{2}, \frac{7}{2}, \frac{11}{2}, \frac{15}{2}, -\frac{19}{2}, \frac{27}{2}) $ & $ 17 $ \\ \hline
    $ 5 $ & $ (-\hf, \frac{3}{2}, \frac{5}{2}, \frac{9}{2}, \frac{13}{2}, \frac{15}{2}, -\frac{11}{2}, \frac{43}{2}) $ & $ (0, 1, 3, 4, 6, 7, -10, 13) $ & $ 13 $ \\ \hline
    $ 6 $ & $ (-\hf, \frac{3}{2}, \frac{5}{2}, \frac{9}{2}, \frac{13}{2}, \frac{17}{2}, -\frac{15}{2}, \frac{5}{2}) $ & $ (0, 1, 3, 4, 6, 7, -10, 13) $ & $ 16 $ \\ \hline
    $ 5 $ & $ (-1, 2, 3, 4, 6, 8, -7, 21) $ & $ (1, 2, 3, 4, 6, 7, -11, 12) $ & $ 15 $ \\ \hline
    $ 6 $ & $ (0, 1, 2, 5, 7, 8, -6, 21) $ & $ (1, 2, 3, 4, 6, 7, -11, 12) $ & $ 14 $ \\ \hline
    $ 5 $ & $ (\hf, \frac{3}{2}, \frac{5}{2}, \frac{9}{2}, \frac{11}{2}, \frac{13}{2}, -\frac{7}{2}, \frac{45}{2}) $ & $ (\hf, \frac{3}{2}, \frac{5}{2}, \frac{7}{2}, \frac{9}{2}, \frac{15}{2}, -\frac{17}{2}, \frac{29}{2}) $ & $ 10 $ \\ \hline
    $ 6 $ & $ (0, 2, 3, 4, 7, 10, -9, 19) $ & $ (\hf, \frac{3}{2}, \frac{5}{2}, \frac{7}{2}, \frac{9}{2}, \frac{15}{2}, -\frac{17}{2}, \frac{29}{2}) $ & $ 19 $ \\ \hline
    $ 5 $ & $ (0, 1, 3, 4, 8, 9, -7, 20) $ & $ (-\hf, \frac{3}{2}, \frac{5}{2}, \frac{7}{2}, \frac{9}{2}, \frac{17}{2}, -\frac{3}{2}, \frac{25}{2}) $ & $ 16 $ \\ \hline
    $ 6 $ & $ (0, 1, 3, 4, 5, 7, -6, 22) $ & $ (-\hf, \frac{3}{2}, \frac{5}{2}, \frac{7}{2}, \frac{9}{2}, \frac{17}{2}, -\frac{3}{2}, \frac{25}{2}) $ & $ 13 $ \\ \hline
    $ 5 $ & $ (-\hf, \frac{3}{2}, \frac{5}{2}, \frac{9}{2}, \frac{15}{2}, \frac{17}{2}, -\frac{13}{2}, \frac{5}{2}) $ & $ (0, 1, 3, 4, 5, 8, -11, 12) $ & $ 15 $ \\ \hline
    $ 6 $ & $ (-\hf, \frac{3}{2}, \frac{5}{2}, \frac{9}{2}, \frac{11}{2}, \frac{15}{2}, -\frac{13}{2}, \frac{43}{2}) $ & $ (0, 1, 3, 4, 5, 8, -11, 12) $ & $ 14 $ \\ \hline
    $ 5 $ & $ (0, 1, 2, 5, 6, 8, -7, 21) $ & $ (0, 1, 2, 5, 6, 7, -11, 12) $ & $ 15 $ \\ \hline
    $ 6 $ & $ (-1, 2, 3, 4, 7, 8, -6, 21) $ & $ (0, 1, 2, 5, 6, 7, -11, 12) $ & $ 14 $ \\ \hline
    $ 5 $ & $ (\hf, \frac{3}{2}, \frac{5}{2}, \frac{7}{2}, \frac{17}{2}, \frac{19}{2}, -\frac{15}{2}, \frac{39}{2}) $ & $ (0, 1, 2, 3, 4, 9, -10, 13) $ & $ 17 $ \\ \hline
    $ 6 $ & $ (\hf, \frac{3}{2}, \frac{5}{2}, \frac{7}{2}, \frac{9}{2}, \frac{13}{2}, -\frac{11}{2}, \frac{45}{2}) $ & $ (0, 1, 2, 3, 4, 9, -10, 13) $ & $ 12 $ \\ \hline
    $ 5 $ & $ (\hf, \frac{3}{2}, \frac{5}{2}, \frac{7}{2}, \frac{11}{2}, \frac{13}{2}, -\frac{9}{2}, \frac{45}{2}) $ & $ (0, 1, 2, 3, 5, 8, -9, 14) $ & $ 11 $ \\ \hline
    $ 6 $ & $ (\hf, \frac{3}{2}, \frac{5}{2}, \frac{7}{2}, \frac{15}{2}, \frac{19}{2}, -\frac{17}{2}, \frac{39}{2}) $ & $ (0, 1, 2, 3, 5, 8, -9, 14) $ & $ 18 $ \\ \hline
    $ 4 $ & $ (\hf, \frac{5}{2}, \frac{7}{2}, \frac{9}{2}, \frac{11}{2}, \frac{13}{2}, -\frac{3}{2}, \frac{45}{2}) $ & $ (0, 1, 2, 3, 5, 6, -7, 16) $ & $ 8 $ \\ \hline
    $ 7 $ & $ (1, 2, 3, 5, 6, 11, -10, 18) $ & $ (0, 1, 2, 3, 5, 6, -7, 16) $ & $ 21 $ \\ \hline
    $ 4 $ & $ (\hf, \frac{3}{2}, \frac{7}{2}, \frac{9}{2}, \frac{11}{2}, \frac{13}{2}, -\frac{5}{2}, \frac{45}{2}) $ & $ (-\hf, \frac{3}{2}, \frac{5}{2}, \frac{7}{2}, \frac{9}{2}, \frac{13}{2}, -\frac{15}{2}, \frac{31}{2}) $ & $ 9 $ \\ \hline
    $ 7 $ & $ (\hf, \frac{3}{2}, \frac{7}{2}, \frac{9}{2}, \frac{13}{2}, \frac{3}{2}, -\frac{19}{2}, \frac{37}{2}) $ & $ (-\hf, \frac{3}{2}, \frac{5}{2}, \frac{7}{2}, \frac{9}{2}, \frac{13}{2}, -\frac{15}{2}, \frac{31}{2}) $ & $ 20 $ \\ \hline
    $ 4 $ & $ (0, 1, 4, 5, 6, 7, -3, 22) $ & $ (-1, 2, 3, 4, 5, 6, -8, 15) $ & $ 10 $ \\ \hline
    $ 7 $ & $ (0, 1, 4, 5, 6, 10, -9, 19) $ & $ (-1, 2, 3, 4, 5, 6, -8, 15) $ & $ 19 $ \\ \hline
    $ 4 $ & $ (0, 1, 3, 4, 5, 6, -2, 23) $ & $ (-\hf, \frac{3}{2}, \frac{5}{2}, \frac{7}{2}, \frac{9}{2}, \frac{15}{2}, -\frac{13}{2}, \frac{31}{2}) $ & $ 8 $ \\ \hline
    $ 7 $ & $ (0, 1, 3, 4, 7, 11, -10, 18) $ & $ (-\hf, \frac{3}{2}, \frac{5}{2}, \frac{7}{2}, \frac{9}{2}, \frac{15}{2}, -\frac{13}{2}, \frac{31}{2}) $ & $ 21 $ \\ \hline
    $ 4 $ & $ (0, 2, 3, 4, 5, 6, -1, 23) $ & $ (0, 1, 2, 3, 5, 7, -6, 16) $ & $ 7 $ \\ \hline
    $ 7 $ & $ (\hf, \frac{3}{2}, \frac{5}{2}, \frac{9}{2}, \frac{13}{2}, \frac{23}{2}, -\frac{3}{2}, \frac{35}{2}) $ & $ (0, 1, 2, 3, 5, 7, -6, 16) $ & $ 22 $ \\ \hline
    $ 4 $ & $ (-1, 2, 3, 4, 5, 6, 0, 23) $ & $ (0, 1, 2, 3, 6, 7, -5, 16) $ & $ 8 $ \\ \hline
    $ 7 $ & $ (0, 1, 2, 5, 6, 12, -11, 17) $ & $ (0, 1, 2, 3, 6, 7, -5, 16) $ & $ 23 $ \\ \hline
    $ 3 $ & $ (\frac{3}{2}, \frac{5}{2}, \frac{7}{2}, \frac{9}{2}, \frac{11}{2}, \frac{13}{2}, -\hf, \frac{45}{2}) $ & $ (0, 1, 2, 3, 4, 5, -6, 17) $ & $ 7 $ \\ \hline
    $ 8 $ & $ (\frac{3}{2}, \frac{5}{2}, \frac{7}{2}, \frac{9}{2}, \frac{11}{2}, \frac{23}{2}, -\frac{3}{2}, \frac{35}{2}) $ & $ (0, 1, 2, 3, 4, 5, -6, 17) $ & $ 22 $ \\ \hline
    $ 3 $ & $ (0, 2, 3, 4, 5, 6, 1, 23) $ & $ (0, 1, 2, 3, 5, 6, -4, 17) $ & $ 7 $ \\ \hline
    $ 8 $ & $ (\hf, \frac{3}{2}, \frac{5}{2}, \frac{9}{2}, \frac{11}{2}, \frac{25}{2}, -\frac{23}{2}, \frac{33}{2}) $ & $ (0, 1, 2, 3, 5, 6, -4, 17) $ & $ 24 $ \\ \hline
    $ 3 $ & $ (0, 1, 3, 4, 5, 6, 2, 23) $ & $ (0, 1, 2, 4, 5, 6, -3, 17) $ & $ 8 $ \\ \hline
    $ 8 $ & $ (0, 1, 3, 4, 5, 13, -12, 16) $ & $ (0, 1, 2, 4, 5, 6, -3, 17) $ & $ 25 $ \\ \hline
    $ 3 $ & $ (1, 2, 3, 4, 5, 6, 0, 23) $ & $ (0, 1, 2, 3, 4, 6, -5, 17) $ & $ 6 $ \\ \hline
    $ 8 $ & $ (1, 2, 3, 4, 6, 12, -11, 17) $ & $ (0, 1, 2, 3, 4, 6, -5, 17) $ & $ 23 $ \\ \hline
    $ 2 $ & $ (0, 1, 2, 4, 5, 6, 3, 23) $ & $ (\hf, \frac{3}{2}, \frac{5}{2}, \frac{7}{2}, \frac{9}{2}, \frac{11}{2}, -\frac{3}{2}, \frac{35}{2}) $ & $ 9 $ \\ \hline
    $ 9 $ & $ (-\hf, \frac{3}{2}, \frac{5}{2}, \frac{7}{2}, \frac{9}{2}, \frac{27}{2}, -\frac{25}{2}, \frac{31}{2}) $ & $ (\hf, \frac{3}{2}, \frac{5}{2}, \frac{7}{2}, \frac{9}{2}, \frac{11}{2}, -\frac{3}{2}, \frac{35}{2}) $ & $ 26 $ \\ \hline
    $ 2 $ & $ (0, 1, 2, 3, 5, 6, 4, 23) $ & $ (0, 1, 2, 3, 4, 5, -1, 18) $ & $ 10 $ \\ \hline
    $ 9 $ & $ (0, 1, 2, 3, 4, 14, -13, 15) $ & $ (0, 1, 2, 3, 4, 5, -1, 18) $ & $ 27 $ \\ \hline
    $ 1 $ & $ (0, 1, 2, 3, 4, 6, 5, 23) $ & $ (0, 1, 2, 3, 4, 5, 1, 18) $ & $ 11 $ \\ \hline
    $ 10 $ & $ (0, 1, 2, 3, 4, 15, -13, 14) $ & $ (0, 1, 2, 3, 4, 5, 1, 18) $ & $ 26 $ \\ \hline
\end{longtable}
\renewcommand\arraystretch{1}

\section{Further Results}

\ms

The following result follows directly from the calculations in the previous sections, which is a generalization of Example 7.2 in \cite{D1} for $E_{6(2)}$, Example 5.2 in \cite{DDL} for $E_{7(-5)}$ and Example 5.2 in \cite{DDW} for $E_{7(7)}$.

\begin{cor}\label{lemma-observation}
Let $G$ be a simply connected, non-Hermitian symmetric simple real Lie group, so that it has a unique minimal $(\frg,K)$-module $V_{min}(\frg_0)$ according to Table \ref{tab:nr-minimal-repns}. Suppose $H_D(V_{min}(\frg_0))\neq0$, then for each $\lambda\in\Lambda(V_{min}(\frg_0))$ such that $\wt{E}_{\lambda-\rho_\frk}$ appears in $H_D(V_{min}(\frg_0))$, it occurs at exactly two Dirac triples $(\mu_{s_1(\lambda)},\rho_1,\lambda)$ and $(\mu_{s_2(\lambda)},\rho_2,\lambda)$. Moreover, $\sigma := s_1(\lambda)+s_2(\lambda)$ is a constant for any choices of $\lambda$.

If in addition $\text{rank}(G)=\text{rank}(K)$, the following statements are equivalent:
\begin{enumerate}
\item
$\DI((V_{min}(\frg_0))=0$;
\item
$\sigma$ is odd;
\item
Every element in $\Lambda((V_{min}(\frg_0))$ is not $\mf{g}$-regular.
\end{enumerate}

\end{cor}

\appendix

\section{Python code used in computation} \label{sec-python}

\begin{rmk}
Some lines are too long in the original code, so we've made some adjustments.
\end{rmk}

\begin{verbatim}

from numpy import *

#some basic functions

def contain(a,b):
    for i in range(len(a)):
        if a[i] == b:
            return True
        else:
            continue
    return False

def inpd(u,v):
    return sum(array(u)*array(v))


def lengthsqr(v):
    return sum(array(v)*array(v))


def refl(w,v):
    return (array(w)-[2*inpd(w,v)/lengthsqr(v)]*array(v)).tolist()



#basic data (for e8(8))

lambda0 = [0,1,1,2,3,4,5,18]

rho0 = [0,1,2,3,4,5,6,23]

rhoc = [0,1,2,3,4,5,6,7]

beta = [0.5,0.5,0.5,0.5,0.5,0.5,0.5,0.5]

mu0 = [0,0,0,0,0,0,0,0]



#basic data (for e8(-24))

lambda0 = [0,1,1,2,3,4,5,18]

rho0 = [0,1,2,3,4,5,6,23]

rhoc = [0,1,2,3,4,5,-8,9]

beta = [0,0,0,0,0,1,0,1]

mu0 = [0,0,0,0,0,0,4,4]





#construct the set of vectors of length n whose coordinates are \pm1.

def z2(n):
    vz2 = [ones(n).tolist()]
    for i in range(n):
        wz2 = []
        for v in vz2:
            w = v[:]
            w[i] *= -1
            wz2.append(w)
        vz2 += wz2
    return vz2



#compute the number of negative coordinates in a vector

def minusnumber(v):
    m = 0
    for i in range(len(v)):
        if v[i] < 0:
            m += 1
    return m



#computing noncompact positive roots (for E8(8))

ncprts = []

for v in z2(8):
    if minusnumber(v) == 0 
    or minusnumber(v) == 2 
    or minusnumber(v) == 4 
    or minusnumber(v) == 6  
    or minusnumber(v) == 8:
        if v[7] > 0:
            ncprts.append((0.5*array(v)).tolist())



#computing noncompact positive roots (for E8(-24))

ncprts = []

for v in z2(6):
    if minusnumber(v) == 0 
    or minusnumber(v) == 2 
    or minusnumber(v) == 4 
    or minusnumber(v) == 6:
        ncprts.append((0.5*array(v+[1,1])).tolist())


for i in range(6):
    for j in [6,7]:
        v = [0,0,0,0,0,0,0,0]
        v[i] = 1
        v[j] = 1
        ncprts.append(v)

for i in range(6):
    for j in [6,7]:
        v = [0,0,0,0,0,0,0,0]
        v[i] = 1
        v[j] = -1
        ncprts.append(v)



#transfering a vector into the k-dominant one of its W(k,t)-orbit (for E8(8))

def kdom(v):
    v1 = []
    for i in range(8):
        v1.append(abs(v[i]))
    v1.sort()
    if minusnumber(v) == 1 
    or minusnumber(v) == 3 
    or minusnumber(v) == 5 
    or minusnumber(v) == 7:
        v1[0] *= -1
    return v1



#transfering a vector into the k-dominant one of its W(k,t)-orbit (for E8(-24))

##taking { } in the k of E7(-5)

def D6dom(v):
    v1 = []
    for i in range(6):
        v1.append(abs(v[i]))
    v1.sort()
    if minusnumber(v) == 1 
    or minusnumber(v) == 3 
    or minusnumber(v) == 5 :
        v1[0] *= -1
    return v1

def A1dom(v):
    v1 = v
    v1.sort()
    return v1

def k1dom(v):
    return D6dom(v[0:6]) + A1dom(v[6:8])

##create a list for all noncompact roots of E7(-5)

ncrte7m5 = []

for v in z2(8):
    if minusnumber(v) == 0 
    or minusnumber(v) == 2 
    or minusnumber(v) == 4 
    or minusnumber(v) == 6 
    or minusnumber(v) == 8:
        if v[6]+v[7] == 0:
            ncrte7m5.append(v)

##turn a vector to E7(-5)-dominant


def g1dom(v):
    v0 = []
    orbit = []
    orbit2 = []
    v0 = k1dom(v)
    for w in ncrte7m5:
        if contain(orbit,k1dom(refl(v0,w))) == False:
            orbit.append(k1dom(refl(v0,w)))
    for u in orbit:
        if inpd(u,[1,-1,-1,-1,-1,-1,-1,1]) > 0:
            return u
        else:
            continue
    for v1 in orbit:
        for w in ncrte7m5:
            if contain(orbit,k1dom(refl(v1,w))) == False 
            and contain(orbit2,k1dom(refl(v1,w))) == False:
                orbit2.append(k1dom(refl(v1,w)))
    for u in orbit2:
        if inpd(u,[1,-1,-1,-1,-1,-1,-1,1]) > 0:
            return u
        else:
            continue    
    
    return None


##turn a vector to k-dominant


def kdom(v):
    v0=[]
    v1=[]
    v0=g1dom(v)
    if v0 != None:
        if v0[6]+v0[7] >= 0:
            return v0
        else:
            v1=refl(v0,[0,0,0,0,0,0,1,1])
            return v1
    else:
        return None




#computing R_K

RK1 = []

RK2 = []

RK3 = []

for i in ncprts:
    v = kdom(refl(rho0,i))
    if contain(RK1,v) == False:
        RK1.append(v)


for i in ncprts:
    for w in RK1:
        v = kdom(refl(w,i))
        if contain(RK2,v) == False 
        and contain(RK1,v) == False:
            RK2.append(v)


for i in ncprts:
    for w in RK1+RK2:
        v = kdom(refl(w,i))
        if contain(RK3,v) == False 
        and contain(RK2,v) == False 
        and contain(RK1,v) == False:
            RK3.append(v)

#..., Repeat until RKi is empty


RK=RK1+RK2#+...

print(len(RK))



#computing \Lambda_K

LK1 = []

LK2 = []

LK3 = []

for i in ncprts:
    v = kdom(refl(lambda0,i))
    if contain(LK1,v) == False:
        LK1.append(v)

for i in ncprts:
    for w in LK1:
        v = kdom(refl(w,i))
        if contain(LK2,v) == False 
        and contain(LK1,v) == False:
            LK2.append(v)

for i in ncprts:
    for w in LK1+LK2:
        v = kdom(refl(w,i))
        if contain(LK3,v) == False 
        and contain(LK2,v) == False 
        and contain(LK1,v) == False:
            LK3.append(v)

#..., Repeat until LKi is empty


LK = LK1+LK2#+...

print(len(LK))



#a function judging the length of w in W(g,t) such that rho=wrho0.

def wlen(r):
    k = 0
    for i in ncprts:
        if inpd(r,i) < 0:
            k+=1
    return k


#looking for Dirac triples, by exhaustion.

def Dirac(s,r,l):
    if kdom((array(mu0)+[s]*array(beta)+array(rhoc)-array(r)).tolist()) 
    == (array(l)-array(rhoc)).tolist():
        return True
    else:
        return False
    
number_of_triples = 0
list_of_l = []

for l in LK:
    for s in range(200):
        for r in RK:
            if Dirac(s,r,l) == True:
                print(s,r,l,wlen(r))
                number_of_triples +=1
                if contain(list_of_l,l) == False:
                    list_of_l.append(l)

print(number_of_triples)
print(len(list_of_l))

\end{verbatim}


\begin{thebibliography}{ALTV}

\bibitem[B1]{B1}
N.~Bourbaki,
{\it Lie Groups and Lie Algebras, Chapter4-6},
Springer-Verlag, Berlin, 2002.

\bibitem[BP1]{BP1}
D.~Barbasch, P. Pand\v zi\'c,
{\it  Dirac cohomology and unipotent representations of complex groups},
Noncommutative geometry and global analysis,
Contemp. Math. \textbf{546}, Amer. Math. Soc., Providence, RI, 2011, pp.~1--22.

\bibitem[BP2]{BP2} D.~Barbasch, P.~Pand\v zi\'c, {\it Dirac cohomology of unipotent representations of
$Sp(2n, \mathbb{R})$ and $U(p, q)$}, J. Lie Theory \textbf{25} (1) (2015), pp.~185--213.

\bibitem[D1]{D1}  C.-P.~Dong,
{\it A non-vanishing criterion for Dirac cohomology}, 
Transformation Groups (2023), to appear (https://doi.org/10.1007/s00031-022-09758-0).

\bibitem[DDL]{DDL}  L.-G.~Ding,  C.-P.~Dong, P.-Y.~Li,
{\it Dirac series of $E_{7(-5)}$}, 
Indagationes Mathematicae, to appear (https://doi.org/10.1016/j.indag.2022.09.002).

\bibitem[DDW]{DDW} Y.-H. Ding, C.-P. Dong, L. Wei
{\it Dirac series of $E_{7(7)}$}, preprint (arXiv:2210.15833).

\bibitem[DDY]{DDY}  J.~Ding, C.-P.~Dong, L.~Yang,
{\it Dirac series for some real exceptional Lie groups}, J. Algebra \textbf{559} (2020), pp.~379--407.

\bibitem[DW1]{DW1}   C.-P.~Dong, K.D.~Wong, {\it Scattered representations of complex classical groups}, Int. Math. Res. Not. IMRN, \textbf{14} (2022), pp.~10431--10457.

\bibitem[H1]{H1}
J.-S. Huang, 
{\it Dirac cohomology, elliptic representations and endoscopy},
Progr. Math., 312
Birkhäuser/Springer, Cham, 2015

\bibitem[HP1]{HP1}
J.-S. Huang, P. Pand\v{z}i\'{c},
{\it Dirac cohomology, unitary representations and a proof of a conjecture of
Vogan}, J. Amer. Math. Soc. \textbf{15} (2002), pp.~185--202.

\bibitem[HP2]{HP2}
J.-S. Huang, P. Pand\v{z}i\'{c},
{\it Dirac Operators in Representation Theory},
Mathematics: Theory and Applications, Birkhauser, 2006.

\bibitem[HPP]{HPP}
J.-S. Huang, P. Pand\v{z}i\'{c}, V. Protsak,
{\it Dirac cohomology of Wallach representations},
Pacific J. Math. \textbf{250} (2011), pp.~163--190.

\bibitem[J1]{J1}
A.~Joseph,
{\it  The minimal orbit in a simple {L}ie algebra and its associated maximal ideal},
Ann. Sci. \'{E}cole Norm. Sup. (4), 1976.

\bibitem[Kn]{Kn}
A.~Knapp,
{\it  Representation Theory of Semisimple Groups},
Princeton University Press, Princeton NJ, 1986.

\bibitem[P1]{P1} R.~Parthasarathy, {\it Dirac operators and the discrete
series}, Ann. of Math. \textbf{96} (1972), pp.~1--30.

\bibitem[P2]{P2} R.~Parthasarathy, {\it Criteria for the unitarizability of some highest weight modules},
Proc. Indian Acad. Sci. \textbf{89} (1) (1980), pp.~1--24.

\bibitem[T1]{T1}
H. Tamori,
{\it Classification of minimal representations of real simple Lie groups}
Math. Z. \textbf{292} (2019), pp.~387--402.

\bibitem[V1]{V1}
D.~Vogan,
{\it  The unitary dual of $GL(n)$ over an archimedean field},
Invent. Math. \textbf{83} (1986), pp.~449--505.

\bibitem[V2]{V2}
D.~Vogan,
{\it Dirac operator and unitary representations},
3 talks at MIT Lie groups seminar, Fall 1997.

\bibitem[V3]{V3}
D.~Vogan,
{\it Singular unitary representations},
Noncommutative harmonic analysis and Lie groups, Springer, 1981.


\bibitem[VZ]{VZ}
D.~Vogan, G.~Zuckerman,
{\it Unitary Representations with non-zero cohomology},
Compositio Math.
\textbf{53} (1984), pp.~51--90.


\bibitem[W1]{W1}
N.R.~Wallach, 
{\it Real reductive groups. {I}},
Academic Press, Inc., Boston, MA, 1988

\end{thebibliography}
\end{document}